\documentclass[12pt]{amsart}
\usepackage[dvipsnames]{xcolor}
\usepackage{soul}
\usepackage{fullpage}
\usepackage{amssymb}
\usepackage{enumitem}
\usepackage{graphicx}  
\DeclareGraphicsExtensions{.pstex,.eps}

\usepackage[obeyspaces,hyphens,spaces]{url}

\usepackage{xcolor}

\renewcommand{\ell}{{l}}  

\newcommand{\R}{{\mathbb{R}}}

\newcommand{\CC}{{\mathbb C}}

\newcommand{\RR}{{\mathbb R}}

\newcommand{\N}{{\mathbb N}}

\newcommand{\tF}{{\tilde{F}}}
\newcommand{\tb}{{\tilde b}}

\newcommand{\p}{\partial} 
\newcommand{\supp}{\operatorname{supp}}

\renewcommand{\Re}{\mathop{\rm Re}\nolimits}
\renewcommand{\Im}{\mathop{\rm Im}\nolimits}

\newcommand{\lX}{l^1 X}
\newcommand{\ltX}{l^2 X}
\newcommand{\ltY}{l^2 Y}
\newcommand{\lY}{l^1 Y}
\newcommand{\lH}{l^1 H}

\renewcommand{\O}{{\mathcal{O}}}

\theoremstyle{plain}

\newtheorem{thm}{Theorem}[section]
\newtheorem{prop}[thm]{Proposition}
\newtheorem{cor}[thm]{Corollary}
\newtheorem{lemma}[thm]{Lemma}

\theoremstyle{definition}

\newtheorem{rem}{Remark}[thm]
\newtheorem{defn}[thm]{Definition} 

\numberwithin{equation}{section}

\def\squarebox#1{\hbox to #1{\hfill\vbox to #1{\vfill}}}

\newcommand{\la}{\langle}
\newcommand{\ra}{\rangle}

\usepackage{amsxtra}

\newcommand{\win}{{w_{in}}}
\renewcommand{\wr}{{w_{R}}}
\newcommand{\wout}{{w_{out}}}

\title[Quasilinear Schr\"odinger equations]
{Quasilinear Schr\"odinger equations III:  Large Data and Short Time}

\author[J.L. Marzuola]
{Jeremy L. Marzuola}

\author[J. Metcalfe]
{Jason Metcalfe}

\author[D. Tataru]
{Daniel Tataru}

\address{Department of Mathematics, University of North Carolina-Chapel Hill \\
Phillips Hall, Chapel Hill, NC  27599, USA}
\email{marzuola@email.unc.edu}

\address{Department of Mathematics, University of North Carolina-Chapel Hill \\
Phillips Hall, Chapel Hill, NC  27599, USA}
\email{metcalfe@email.unc.edu}

\address{Mathematics Department, University of California \\
Evans Hall, Berkeley, CA 94720, USA}
\email{tataru@math.berkeley.edu}

\begin{document}

\begin{abstract}
  In this article we prove short time local well-posedness in low-regularity
  Sobolev spaces for large data general quasilinear Schr\"odinger
  equations with a nontrapping assumption. These results represent improvements over the small data regime considered by the authors in \cite{MMT3,MMT4}, as well as the pioneering  works by Kenig-Ponce-Vega and Kenig-Ponce-Rolvung-Vega~\cite{KPV,KPRV1,KPRV2}, 
where viscosity methods were used to prove existence
of solutions  for localized data in  high regularity spaces.  Our arguments here are
purely dispersive.  The function spaces in which we show
existence are constructed in ways motivated by the results of
Mizohata, Ichinose, Doi, and others, including the authors.
\end{abstract}

\maketitle

\section{Recap of History and Discussion of 
Main Results}

In this article we consider the large data local well-posedness for
quasilinear Schr\"odinger equations, extending the earlier small data
results of the authors in \cite{MMT3,MMT4}.  Specifically, we will
study equations of the form
\begin{equation}
\label{eqn:quasiquad}
\left\{ \begin{array}{l}
i u_t + g^{jk} (u,\bar u, \nabla u,\nabla \bar u ) \p_j \p_ku = 
F(u,\bar u,\nabla u, \nabla \bar u) , \quad u:
\RR \times \RR^d \to \CC^m \\ \\
u(0,x) = u_0 (x).
\end{array} \right.
\end{equation}
Here $g$ and $F$ are assumed to be smooth functions of their variables, with $g$ real and positive definite. In particular this allows for them to depend on both $u$ and $\bar u$.

In the small initial data case these equations have already been considered by the authors
 \cite{MMT3,MMT4}  in  spaces of relatively low Sobolev regularity. The latter paper 
considers problems which have only cubic and higher nonlinearities, where the 
initial data is in Sobolev spaces $H^s$. The former paper allows for general nonlinearities,  i.e. including quadratic terms,  but with smaller Sobolev spaces $\ell^1 H^s$, which are still translation invariant but have some stronger summability assumptions.  These will be made precise in the discussion below. 

The aim of this paper is to prove instead a local well-posedness result for the large data problem. Compared with the small data case,  here we need to contend with an
additional obstacle, namely trapping.  To prevent this,  we impose a nontrapping 
condition on the initial data. Then, as a part of our results, we prove that 
nontrapping persists for some small time. In this context, the  lifespan of the solutions no longer depends only on the data size. Instead, our lower bound on the lifespan will also depend on a quantitative form of the nontrapping assumption.

Here we will work primarily with quadratic nonlinear interactions, as
in \cite{MMT3}. We will also state the counterpart of the result in the cubic case,
as considered for  small data in \cite{MMT4};
however, as the proofs only differ slightly we will focus on the quadratic
setting and only remark where the proofs need adjustment for the cubic
interactions.  Specifically, we will study the equations \eqref{eqn:quasiquad}
assuming that
\[
g : \CC^m\times\CC^m \times (\CC^m)^d \times (\CC^m)^d \to
\RR^{d \times d}, \qquad 
F: \CC^m\times \CC^m \times (\CC^m)^d \times (\CC^m)^d \to \CC^m
\]
are smooth functions which,  for $y,z \in \CC^m \times (\CC^m)^d$, satisfy
\begin{equation}\label{gF}
g(0) = I_d, \qquad 
|F(y,z)| \sim O(|y|^2+|z|^2) \text{ near } (y,z) = (0,0)
\end{equation}
in the quadratic interaction problem and
\begin{equation}\label{gT}
g(y,z) = I_d+O(|y|^2+|z|^2), \qquad 
F(y,z)= O(|y|^3+|z|^3) \text{ near } (y,z) = (0,0).  
\end{equation}
in cubic interaction problem.
We will also assume uniform ellipticity of $g$ in both cases.  Namely, we will assume that 
\[
c_0 |\xi|^2 \leq g^{jk} \xi_j \xi_k \leq c_0^{-1} |\xi|^2
\] 
for a fixed $c_0 >0$.

As in \cite{MMT3,MMT4}, we also consider a second class of  quasilinear
Schr\"odinger equations 
\begin{equation}
\label{eqn:quasiquad1}
\left\{ \begin{array}{l}
i u_t + \p_j g^{jk} (u,\bar u)  \p_ku = 
F(u,\bar u,\nabla u,\nabla \bar u) , \ u:
\RR \times \RR^d \to \CC^m \\ \\
u(0,x) = u_0 (x),
\end{array} \right. 
\end{equation}
with $g$ and $F$ as in \eqref{gF} but where the metric $g$ depends on
$u$ but not on $\nabla u$.  Such an equation is obtained for instance
by differentiating the first equation
\eqref{eqn:quasiquad}. Precisely, if $u$ solves \eqref{eqn:quasiquad}
then the vector $(u,\nabla u)$ solves an equation of the form
\eqref{eqn:quasiquad1}, with a nonlinearity $F$ which depends 
at most quadratically on $\nabla u$.

We note that the second order operator in \eqref{eqn:quasiquad1} is
written in divergence form. This is easily achieved by commuting the
first derivative with $g$ and moving the output to the right hand
side. However, the second order operator in \eqref{eqn:quasiquad}
cannot be written in divergence form without possibly changing the type
of the equation.
 
The proof of the large data result presented here follows the the same
strategy as in the works \cite{MMT3,MMT4}. The main novelty in this
paper is in the proof of the local energy decay estimates for the
linearized equation. The difficulty is not only that we need to work
with a large nontrapping metric, but also that, in order to prevent a
nonlinear energy cascade to high frequencies, we have to produce a
very accurate bound for the (exponentially large) constant in the
local energy bounds in terms of our quantitative nontrapping
parameters. Our proof requires a new multiplier construction for the estimates 
since here we assume no quantitative decay of the solution in the physical space.
Also, a careful set-up and ordering of large constants is essential in order 
to avoid a circular argument.

Even at the linear level, an obstruction to well-posedness comes from
the infinite speed of propagation phenomena.  We recall some of the
issues here as it is even more relevant for the large data problem.
From \cite{Ich,MMT1,Miz1,Miz2,Miz3,Tak80}, it is known that even in the case
of linear problems of the form
\begin{equation} \label{lin}
(i \partial_t  +  \Delta_g)  v = A_i(x) \partial_i v,
\end{equation}
a necessary condition for $L^2$ well-posedness is an integrability
condition for the (imaginary part of) the magnetic potential $A$ along
the Hamilton flow of the leading order differential operator.  In the
case of \eqref{eqn:quasiquad}, we would have to look instead at the
corresponding linearized problem, which would exhibit a magnetic
potential of the form $A = A(u,\nabla u)$.  If one considers equations
with quadratic terms and with $H^s$ initial data, then such a
potential does not generically satisfy Mizohata's integrability
condition. Thus, some further decay condition on the initial data is
necessary. A further motivation for decay conditions comes from the
large data problem, where one seeks to confine the trapping analysis
to a compact set. Together, these two observations show that in our
context nontrapping is a compact phenomena, and also that it is stable
with respect to small perturbations of $u_0$.

Indeed, such a decay condition was manifest in the seminal papers
\cite{KPV,KPRV1,KPRV2}, where the first local well-posedness results
for this problem were obtained. There, local well-posedness results
for this problem were proved for data (and solutions) in $H^s \cap L^2
(\langle x \rangle^N)$, where $\langle x \rangle=(1+|x|^{2})^{\frac{1}{2}}$, for some 
large $s$ and $N$.

By contrast, our previous results in \cite{MMT3,MMT4} apply for
initial data in spaces which are not only low regularity, but also
translation invariant.  Precisely, for the quadratic problem
\eqref{gF} we use the $\ell^1 H^s$ spaces (see a
precise definition below), while for the cubic problem \eqref{gT} we revert
to the classical $H^s$ spaces.  Maintaining this natural setting
is one of our objectives for the large data problem.

We now turn our attention to the nontrapping condition for the initial
data.  This is defined in a qualitative manner in terms of the
Hamilton flow associated with $g(u_0)$ (or equivalently, the geodesic flow
associated to the Riemannian metric $g(u_0)$):
\begin{defn}
We say that the metric $g(u_0)$ is nontrapping if all nontrivial bicharacteristics for $\Delta_{g(u_0)}$ escape to spatial infinity at both ends.
\end{defn}

To motivate the fact that this definition is meaningful for $u_0$ in
our initial data spaces we make several observations. Firstly, our
choice for the initial data space guarantees that $g(u_0) \in C^2$,
therefore its Hamilton flow is well defined locally.  Secondly, while
$u_0$ may be large, the fact that we are using $L^2$ based spaces
implies that $u_0$ is small in our function spaces outside a large
compact set. Thus trapping is necessarily confined to
bicharacteristics which intersect this compact set.

The above qualitative definition of trapping suffices in order to
state our main results. However, in order to prove the results, as
well as to provide a lifespan bound, we will have to use a more
precise quantitative characterization of nontrapping.

Before stating our main results, we briefly recall the definition of the $\ell^1 H^s$
spaces, following \cite{MMT3}. These are defined using a standard spatial Littlewood-Paley decomposition
\[
1 = \sum_{k \in \N} S_k
\]
where $S_0$ selects all frequencies of size $\lesssim 1$. Corresponding to each dyadic 
frequency scale $2^j \geq 1$ we consider an associated  partition $\mathcal{Q}_{j}$ of $\RR^{d}$ into cubes of side length $2^{j}$ and an associated smooth partition of unity
\[
1 = \sum_{Q \in \mathcal{Q}_{j}}  \chi_Q.
\]
Then we can define the $l^1_j L^2$ norm by 
\[
\| u\|_{l^1_j L^2} = \sum_{Q \in \mathcal{Q}_{j}} \| \chi_Q u\|_{L^2},
\]
and the space $l^1 H^s$ with norm given by
\[
\| u\|_{l^1 H^s}^2 = \sum_{j \geq 0} 2^{2sj} \|S_j u\|_{ l^1_j L^2}^2 .
\]
With our spaces in hand, we can now state our main result concerning the quasilinear problem \eqref{eqn:quasiquad} with data $u_0 (x) \in l^1 H^s$ and quadratic interactions \eqref{gF}.

\begin{thm}
\label{thm:main1}
a) Let $s > \frac{d}2+3$.  Let $u_0 \in l^1 H^s$ be a nontrapping
initial datum for the equation \eqref{eqn:quasiquad} with quadratic
interactions \eqref{gF}.   Then, there exists $T= T(u_0) > 0$ sufficiently small
such that the equation \eqref{eqn:quasiquad} is locally well-posed in
$l^1 H^s (\RR^d)$ on the time interval $I = [0,T]$.

b) The same result holds for the equation \eqref{eqn:quasiquad1}
with $s > \frac{d}2 + 2$.
\end{thm}

The counterpart of this result for the cubic problem \eqref{gT} is as follows using standard Sobolev spaces, which extends \cite{Michalowski} to the low regularity regime:
\begin{thm}
\label{thm:main2}
a) Let $s > \frac{d+5}2$. Let $u_0 \in H^s$ be a nontrapping
initial datum for the equation \eqref{eqn:quasiquad} with cubic
interactions \eqref{gT}.   Then, there exists $T= T(u_0) > 0$ sufficiently small
such that the equation \eqref{eqn:quasiquad} is locally well-posed in
$l^2 H^s (\RR^d)$ on the time interval $I = [0,T]$.

b) The same result holds for the equation \eqref{eqn:quasiquad1} with cubic nonlinear interactions
with $s > \frac{d+3}2 $.
\end{thm}

\begin{rem}
  The well-posedness result in the statement of the theorems above has
  to be interpreted in the classical quasilinear fashion. Precisely,
  in the setting of Theorem~\ref{thm:main1}, it asserts that

\begin{enumerate}[itemsep=1ex,label=(\roman*)]
\item (Regular solutions) For large $\sigma$ and regular nontrapping
data $u_0 \in \ell^1 H^\sigma$ there exists a unique local nontrapping
solution $u \in C(\ell^1 H^\sigma)$ on a nonempty maximal time
interval $I = [0,T_{max}(u_0))$.

\item (Rough solutions) For $s$ as in the theorem and nontrapping data
$u_0 \in \ell^1 H^s$ there exists a unique local nontrapping
solution $u \in C(\ell^1 H^s) \cap \ell^1 X^s$ on a nonempty maximal time
interval $I = [0,T_{max}(u_0))$.  The spaces $l^1 X^s$, which capture
the space-time local energy decay structure, will be described in the next section. 

\item (Continuous dependence) The maximal time $T_{max}(u_0)$  is a lower semicontinuous 
function of $u_0$ in  the $\ell^1 H^s$ topology, and the data to solution map $v_0 \to v$
is continuous at $u_0$ from $\ell^1 H^s$ into $C([0,T];\ell^1 H^s) \cap \ell^1 X^s([0,T])$ 
for all $T < T_{max}(u_0)$. 
\end{enumerate}

The latter property allows one to alternately uniquely identify rough solutions as limits of regular solutions.
\end{rem}

\begin{rem}
  The existence time $T$ in the theorem is allowed to depend on the
  full profile of the initial data $u_0$, and not only on its
  size. This dependence will be made more clear later on. The
  difficulty is that the well-posedness depends not only on the size
  of the data, but also on the nontrapping property of the
  metric. Later we will introduce a more explicit parameter $L=L(u_0)$,
  which quantitatively measures nontrapping, and is stable 
with respect to small perturbations of $u_0$.  Then we will show that
  it suffices to choose $T \ll 1$ so that
\begin{equation}
|\log T| \gg_{\|u_0\|_{l^1 H^s}} L
\end{equation}
where the implicit dependence on the data size is polynomial.
\end{rem}

\begin{rem}

In order to define nontrapping, we require a well-defined Hamilton flow.
A sufficient condition which guarantees this is to have a $C^{1,1}$ metric,
and that is satisfied in the context of Theorem~\ref{thm:main1}. However, 
this is not guaranteed in the case when $s<\frac{d}{2}+2$, in Theorem \ref{thm:main2}.   Nevertheless,  we can instead guarantee $\nabla^2 g \in L^2(\R_{x_1}, L^\infty(\R^{d-1}_{x'}))$, as well as in any other rotated frame; this still suffices 
in order to define the flow as a bi-Lipschitz map. In turn, implementing this would require appropriate changes in Section~\ref{sec:LED1}; we omit these, and instead
refer the interested reader to \cite{Blair} where a similar analysis was conducted.

\end{rem}

We further remark that the short time large data result
cannot be obtained by scaling from the small data result. This is due
to the fact that the spaces used are inhomogeneous Sobolev spaces, and
spatial localization is not allowed due to the infinite speed of
propagation.  A reflection of this is the fact that the small data case is nontrapping, while in  the large data regime, one must also take into account
the existence of trapping.

The paper is organized as follows. In Section~\ref{sec:boot} we describe the space-time
function spaces in which we will solve \eqref{eqn:quasiquad} and
\eqref{eqn:quasiquad1}; these are identical to those in the small data
setting in \cite{MMT3,MMT4}.  In Section \ref{outline} we introduce some key notations,
including the main size parameters which govern our lifespan bound, and
give an overview of the proof. Section \ref{sec:mult} contains the
necessary multilinear and nonlinear estimates in order to close the
eventual bootstrap estimates; some of these are from \cite{MMT3,MMT4},
but the main bound for the paradifferential error term is new.
The stability of the nontrapping assumption will be discussed in Section~\ref{nontrap};
this is critical in order to propagate nontrapping to positive times.
In Section \ref{sec:LED1} we establish local energy decay for a
linear, nontrapping, inhomogeneous paradifferential version of the Schr\"odinger
equation.  Finally, in Section \ref{sec:proof}, we combine the above
estimates with the proper paradifferential decomposition of the
equation in order to conclude the proof of Theorem~\ref{thm:main1}.

\bigskip

{\sc Acknowledgments.} The first author was supported in part by
  U.S. NSF Grant DMS--1312874 and NSF CAREER Grant DMS-1352353.
   The third author is supported in part by the NSF grant
  DMS-1800294 as well as by a Simons Investigator grant from the
  Simons Foundation.  The authors also wish to thank 
  the Mathematical Sciences Research Institute for hosting two of them during part of this work.  We thank the anonymous referees for a careful reading of the result and for making helpful suggestions to improve the exposition.

\section{Recap of Function Spaces and Notations}
\label{sec:boot}

In this section we recall the definition of the main function spaces 
as well as some of their key properties. For this we follow our previous works, \cite{MMT3,MMT4}. 

We will use an inhomogeneous  Littlewood-Paley decomposition $\sum_{k\in \N} S_k=1$. We set $u_j = S_j u$ and
\[S_{\le N} f = \sum_{i=0}^N S_i f,\quad S_{\ge N}f = \sum_{i=N}^\infty S_i f, \quad S_{[N_1,N_2]} f =\sum_{i=N_1}^{N_2}S_i f.\]
When it is clear from the context, we may abuse notation and use $u_0 = S_0 u$.

Given a translation invariant Sobolev-type space $U$, we define $l^p_j U$ via
\begin{eqnarray*}
\| u \|_{l^p_j U}^p =  \sum_{Q \in \mathcal{Q}_j} \| \chi_Q u
  \|_U^p,  
\end{eqnarray*}
which generalizes the notion of $l^1_j L^2$ defined in the introduction.  We make the natural modification when $p=\infty$.  Upon replacing the sum over the cubes with an integral, the existence of translation invariant norms that are equivalent can easily be checked.
As noted previously in \cite{MMT3}, the smooth partition of compactly supported cutoffs
in the $l^1_j U$ spaces can be replaced by cutoffs which are frequency
localized when it is convenient. 

We recall briefly the local energy type space $X$
of functions on $[0,T] \times \R^d$, with norm
\[
  \| u \|_{X}  =   \sup_{l} \sup_{Q \in \mathcal{Q}_l} 
2^{-\frac{l}{2}} \| u \|_{L^2_{t,x} ([0,T] \times Q)}.
\]
Note, these spaces are dependent upon $T$, and hence the dependence upon the time interval in the estimates below will be treated with some care.

We define $Y\subset L^2_{t,x}([0,T] \times \R^d)$ to satisfy $X=Y^*$.  This will be the space in which we measure the forcing terms for the Schr\"odinger equation.  See, e.g., \cite{BRV}, \cite{MMT3} for more details on the construction of this atomic space.

We set
\[
X_j = 2^{-\frac{j}2} X \cap L^\infty L^2,
\]
which will incorporate the half-degree of smoothing into our local energy spaces.  We then add the $l^p$ spatial summation on the $2^j$ scale
to obtain the space $l^p_j X_j$ with norm
\begin{eqnarray*}
  \| u \|^p_{l^p_j X_j} = 
\sum_{Q \in \mathcal{Q}_j} \| \chi_Q u \|^p_{X_j}.
\end{eqnarray*}
We finally define the spaces $l^p X^s$ by
\begin{equation}
\label{l1Xdef}
  \| u \|^2_{l^p X^s}  =  \sum_j 2^{2js} \| S_{j} u \|^2_{l^p_j X_j}.
\end{equation}
For quadratic interactions, we shall use $\lX^s$ to bound the high frequencies exterior to a large ball in our solutions to \eqref{eqn:quasiquad},
\eqref{eqn:quasiquad1} with nontrapping $l^1 H^s$ data.  For cubic interactions, we shall use $p=2$ and data in $l^2 H^s\approx H^s$.

For the inhomogeneous terms at frequency $2^j$, we shall use
\[
Y_j = 2^{\frac{j}2} Y + L^1 L^2
\]
which has norm
\[
\| f \|_{Y_j} = \inf_{f =  2^{\frac{j}2} f_1 + f_2} \|f_1\|_{Y} + 
\|f_2\|_{L^1 L^2}.
\]
We then similarly consider
\begin{eqnarray}
\label{l1Ydef}
  \| f \|^2_{l^p Y^s} =  \sum_j  2^{2js} \| S_{j} f \|^2_{l^p_j Y_j},
\end{eqnarray}
where $p=1$ will be utilized for the case of quadratic interactions and $p=2$ for the cubic case.

We also record the spaces $X^s$ without the summability. These are given by the norm
\begin{equation}
\label{Xdef}
  \| u \|^2_{X^s}  =  \sum_j 2^{2js} \| S_{j} u \|^2_{X_j}.
\end{equation}
Similarly, we define $Y^s$ via
\begin{equation}
\label{Ydef}
  \| f \|^2_{Y^s}  =  \sum_j 2^{2js} \| S_{j} f \|^2_{Y_j}.
\end{equation}

In the regimes where we can apply paradifferential analysis, it is convenient to present
our bilinear and nonlinear estimates using the method of frequency
envelopes, which we recall below. For a Sobolev-type space $U$ so that
\[
\|u\|_{U}^2 \sim \sum_{k=0}^\infty \|S_k u\|_{U}^2 
\]
 a positive sequence $c_j$  is called an admissible frequency envelope for $u$ in $U$ provided that it
\begin{enumerate}
\item controls the dyadic $U$ size,
\[
\| S_k u\|_{U} \lesssim c_k,
\]
\item is controlled by the $U$ norm,
\[
\sum_{k \in \N} c_k^2 \lesssim \|u\|_{U}^2
\]

\item is slowly varying to the left,
\[
c_j \geq 2^{\delta(j-k) } c_k, \qquad j < k,
\]
and 
\item is uniformly varying to the right,
\[
c_j \geq 2^{\sigma (k-j) } c_k, \qquad j > k
\]
for a fixed (large) $\sigma$.
\end{enumerate}
These properties are easily adapted to the case when the $l^2$
dyadic summability is replaced by $l^p$ with $1 \leq p < \infty$.

An admissible frequency envelope 
always exists, say by 
\begin{equation}\label{freqEnv}
c_j =  \max_{k > j} 2^{-\delta |j-k|} \|S_k u\|_{U} +  \max_{k < j} 2^{-\sigma |j-k|} \|S_k u\|_{U} .
\end{equation}
In the sequel we will use frequency envelopes for the spaces $l^pH^s$, $l^pX^s$ and $l^pY^s$ for $p=1,2$.


\section{Outline of the Proof}
\label{outline}

Let us briefly outline the ideas we will pursue below for the case of
quadratic interactions and the equation \eqref{eqn:quasiquad1}.  The cubic case will follow similarly.  We seek to solve  the equation 
\begin{equation}
\label{eqn:quasiquad1-re}
\left\{ \begin{array}{l}
i u_t + \p_j g^{jk} (u,\bar u)  \p_ku = 
F(u,\bar u,\nabla u,\nabla \bar u) , \ u:
\RR \times \RR^d \to \CC^m, \\ \\
u(0,x) = u_0 (x),
\end{array} \right. 
\end{equation}
in $\ell^1 H^s$ for $s > s_0 > \frac{d}{2} + 2$.
Here we have explicitly included the dependence upon $u$ and $\bar u$ separately since the large data dynamics will depend upon them both more delicately than  in the small data case.  
\bigskip

{\bf 1. The linearized and paradifferential equation.}
An important role in the analysis will be played by the linearized equation,
which has the form
\begin{equation}
\label{eqn:quasi-lin}
\left\{ \begin{array}{l}
i v_t + \p_j g^{jk}  \p_kv + b^j \partial_j v +
\tilde b^j \partial_j \bar v + c  v 
+ \tilde c \bar v = 0,
 \\ \\
v(0,x) = v_0 (x),
\end{array} \right. 
\end{equation}
where the coefficients $g$, respectively $b^j$, $\tilde b^j$, $c$,
$\tilde c$, are smooth nonlinear expressions in $u$, respectively
$u,\nabla u$, which can be explicitly calculated in terms of $g$ and
$F$.   In particular, we have
\begin{equation} \label{bc-def}
\begin{split}
b^j = \partial_u g^{jk} \partial_k u  - \partial_{(\nabla
  u)_j} F, \qquad  {\tilde b}^j = \partial_{\bar u} g^{jk} \partial_k u - \partial_{(\nabla \bar u)_j} F,
\\
c =\partial_j (\partial_u g^{jk}) \partial_k u
- \partial_{ u} F ,
 \qquad  {\tilde c}  =  \partial_j (\partial_{\bar u} g^{jk}) \partial_k u - \partial_{ \bar u} F
\end{split}
\end{equation}
for a fixed $u \in l^1 X^s$.

From the linearized equation we extract its associated 
linear paradifferential flow
\begin{equation}
\label{eqn:quasi-para}
\left\{ \begin{array}{l}
i \partial_t w + \p_j T_{g^{jk}}   \p_k w + T_{b^j} \partial_j  w + 
T_{\tilde b^j} \partial_j \bar w = \tilde f,
 \\ \\
w(0,x) = w_0 (x),
\end{array} \right. 
\end{equation}
where we take the paraproduct operator to be
\begin{equation}
\label{paraproduct}
T_a b = \sum_{N \geq 4} S_{\leq N-4} a S_N b.
\end{equation}
For some of the analysis it will be more convenient  to use the Weyl quantization 
for the paraproduct, which is denoted as follows:
\[
T^w_a b = OP^w\Bigl(\sum_{N \geq 4} S_{\leq N-4} a(x) s_N (\xi)\Bigr).
\]

If $v$ solves \eqref{eqn:quasi-lin} then it also solves \eqref{eqn:quasi-para} with
\[
\tilde f = (\p_j T_{g^{jk}} (u,\bar u)  \p_k w - \p_j g^{jk}  \p_k w ) + (T_{b^j} \partial_j  w  - b^j \partial_j w) + ( T_{\tilde b^j} \partial_j \bar w - \tilde b^j \partial_j \bar w) - cw - \tilde c \bar w.
\]
Here heuristically $\tilde f$ contains only high-high frequency interactions; e.g. in the multilinear case, the two highest frequencies must always\footnote{ This is not entirely accurate in what we do. Instead, $\tilde f$
is also allowed to contain low-high interactions as long as the high frequency factors are undifferentiated; this makes such terms perturbative.}
be balanced. Because of this, its contribution will always be treated perturbatively.

\bigskip

{\bf 2. Rewriting the equation.}  
Based on the expressions above for the linearized equation and its paradifferential truncation,
we write the full nonlinear equation \eqref{eqn:quasiquad1-re} in a paradifferential form,
namely 
\begin{equation}
\label{eqn:quasi-para1}
\left\{ \begin{array}{l}
i \partial_t u + \p_j T^w_{g^{jk}}  \p_k u + T^w_{b^j} \partial_j  u + 
T^w_{\tilde b^j} \partial_j \bar u =  G,
 \\ \\
u(0,x) = u_0 (x).
\end{array} \right. 
\end{equation}
Here the nonlinearity $G = G(u,\bar u, \nabla u, \nabla \bar u)$ is no longer purely algebraic, as it involves 
frequency localizations. The key idea in our proof of the local well-posedness result is that $G$ plays 
a perturbative role.  So we have it, we record that 
\begin{equation*}
G  (u, \bar u , \nabla u , \nabla \bar{u})  = F(u, \bar u , \nabla u , \nabla \bar{u}) - \partial_j (g^{jk}  - T^w_{g^{jk}}) \partial_k u + T^w_{b_j} \partial_j u + T^w_{\tilde b_j} \partial_j \bar u 
\end{equation*}
where $b$ and $\tilde b$ contain both contributions arising from $g$ and from $F$, see \eqref{bc-def}.

The solutions will be constructed via an iterative scheme, where 
we set $u^{(0)} = 0$, and successively define $u^{(n+1)}$ as the solution to
the linear equation
\begin{equation}
\label{eqn:nliter}
\left\{ \begin{array}{l}
 \left(i \p_t +   \p_j T^w_{g^{jk,(n)}} \p_k + T^w_{b^{(n)}}\cdot  \nabla 
\right) u^{(n+1)} +    T^w_{\tilde b^{(n)}} \cdot \nabla  \bar u^{(n+1)} = G(u^{(n)}, \nabla u^{(n)}),
 \\ \\
 u^{(n+1)}(0)  = u_{0}
\end{array} \right. 
\end{equation}
where
\[
g^{(n)} = g(u^{(n)}), \qquad b^{(n)} = b(u^{(n)}, \nabla u^{(n)}), \qquad \tilde b^{(n)} = \tilde b(u^{(n)}, \nabla u^{(n)}).
\]
Each of $G, g^{(n)}, b^{(n)}, \tilde{b}^{(n)}$ also depends on the conjugates of the solution, but this is suppressed here.
In order to guarantee the convergence of this scheme we will carefully
choose time $T$ small enough, depending on the initial data
profile.

\bigskip  
 
{\bf 3. Quantifying nontrapping: the parameters $M,R,L$.}  For the purpose of characterizing
the nontrapping properties of the metric $g(u_0)$ we do not need the
full $H^s$ regularity. Instead we will use a smaller exponent $s_0$
so that
\begin{equation} \label{s0}
\frac{d}2 + 2 < s_0 < s.
\end{equation}
 Its
choice within these bounds is not important, but we fix it once and
for all. The gap between $s_0$ and $s$ will be critical in order to
propagate the nontrapping property.

The first parameter $M$  we use to describe nontrapping 
measures the size of the data,
\begin{equation}\label{M}
M = \| u_0\|_{\lH^{s_0}}.
\end{equation}

Outside a compact spatial region $B= B(x_0,R)$ where both $x_0$ and $R$ depend on $u_0$,
the metric $g_{ij}(u_0)$ will have a small $l^1 H^{s_0}$ norm, and
thus be nontrapping.  The nontrapping assumption guarantees that all
geodesics intersecting $B$ will eventually leave $2B$ at both ends, and never return to 
$B$ once leaving $2B$. We denote by $L \gtrsim R$ the maximum Euclidean length of any such geodesic within $2B$.   The nontrapping condition is then shown to be stable with respect to perturbations, $g \to g + \delta g$,
of the metric which satisfy an exponential smallness condition
\begin{equation}\label{small-pert}
\| \delta g\|_{l^1 X^{s_0}} \lesssim e^{-C_0(M) L}.
\end{equation}
This is proved in Section~\ref{nontrap}.

\bigskip

{\bf 4.  Nontrapping and norm inflation.}
In order to carry out the above iteration, we need to consider 
energy estimates and local energy decay for  the linear paradifferential flow 
\eqref{eqn:quasi-para}.
Precisely, we would like to have bounds of the form
\begin{equation}\label{le}
\| w\|_{l^1 X^{\sigma}} \leq C (\|w_0\|_{\ell^1 H^\sigma} + \|f\|_{l^1 Y^\sigma}), \qquad 0 \leq \sigma 
\end{equation}
in a time interval $[0,T]$ where $T$ depends only on the initial data
$w_0$.

However, even  if $w_0 = 0$, $\sigma = 0$ and $T$ is arbitrarily small, the energy bounds for this system will 
exhibit $L^2$ growth,  due both to the large metric in the compact set $B(x_0,R)$ and to the large coefficient $b$ for the first order term. 
Hence, the best we could hope for is a bound of the type
\begin{equation}\label{le0}
\| w\|_{X^0} \lesssim e^{C(M) L} (\|w_0\|_{L^2} + \|f\|_{Y^0}),
\end{equation}
with the only redeeming feature that only a lower regularity $\lX^{s_0}$ bound for $u$, which occurs in the coefficients,
is needed. Naturally, the constant $C$ in \eqref{le} would have to
be at least as large as the exponential in \eqref{le0}.  Indeed, in
Section~\ref{sec:LED1}  we establish that the bound \eqref{le0} holds.

The key to handle  this exponential growth is to restrict to a very short 
time interval $[0,T]$, with $T$ satisfying
\begin{equation}
T \lesssim e^{-C(M) L}.
\end{equation}
To balance the choice of $T$ and the exponential growth we divide and
conquer.  We first prove a high frequency energy estimate via positive
commutator methods that control the high frequencies while allowing low
frequency errors; this part is independent of the length of the time
interval.  Then, we are able to use the short time in a more direct
fashion to control the contribution from the the low frequencies.

\bigskip

{\bf 5. Nontrapping and high frequency energy estimates.}  The first
step in the proof of \eqref{le0} is to use a positive commutator
method in order to establish local energy bounds with low frequency errors,
\begin{equation}\label{le-high}
\|w\|_{X^0} \lesssim e^{C(M) L} (\|w_0\|_{L^2} + \|f\|_{Y^0} + \|w\| _{L^2L^2})
\end{equation}
with no restriction on the time $T$. This is done in three stages:

(i) Bounds for incoming rays. Here we estimate the energy along
geodesics which approach the compact region $2B$ without any norm
inflation by introducing a suitable incoming multiplier $Q_{in}$.

(ii) Bounds in a compact set $2B$. This is where we bound the local
energy norm of the solution in $2B$ in terms of the incoming part,
using the nontrapping condition to construct a suitable multiplier
$Q_{comp}$. This is where the norm inflation occurs.

(iii) Global bounds. Here we use the local energy estimate in $2B$ in order to 
produce a global exterior bound, which follows very similarly to the small data metric perturbation theoretic arguments of \cite{MMT3,MMT4}.

\bigskip

{\bf 6. Low frequency estimates for short time.}
The last step in the proof of \eqref{le0}  is to complement the above high frequency bound
 with an estimate for the  low frequencies.  This is quite trivial and is obtained by 
H\"older's inequality in time, which gives
\begin{equation} \label{le-low}
\|w\|_{L^2 L^2} \lesssim T^\frac12   \| w\|_{L^\infty L^2} \lesssim T^\frac12   \| w\|_{X^0}.
\end{equation}

\bigskip

{\bf 7. Uniform bounds for the iteration scheme.} 
The difficulty we face here is that, in view of the bound \eqref{le0}, 
the best we can expect of the sequence $u^{(n)}$ in $\ell^1 X^s$ is a bound 
with exponential growth of the  form
\begin{equation} \label{s-bd}
\| u^{(n)} \|_{\ell^1 X^s} \lesssim  e^{C(M) L} \| u_0\|_{\ell^1 H^s} .
\end{equation}
Such a bound cannot be directly obtained in a self-contained inductive argument
and is predicated on the additional assumption that 
\[
\| u^{(n)} \|_{\ell^1 X^{s_0}} \lesssim M.
\]

To avoid a circular argument, we will obtain this last bound not directly from \eqref{s-bd},
but rather by interpolating with a lower regularity  bound but which has a $T$ factor,
\[
\| u^{(n)} - u_0 \|_{\ell^1 L^\infty L^2} \lesssim T M.
\]
After interpolation, we will be able to leverage the remaining (small) power of $T$ 
against the exponential provided that
\begin{equation}
T \ll_{M_s} e^{-C(M) L}, \qquad M_s =  \| u_0\|_{\ell^1 H^s} .
\end{equation}

\bigskip

{\bf 8. Lipschitz bounds and weak convergence for the iteration scheme.}
Here we consider the difference equations for $u^{(n+1)} - u^{(n)}$ and use the 
bound \eqref{le0} directly to prove the convergence of the iteration scheme in 
$\ell^1 X^0$. Given the uniform $\ell^1 X^s$ bound, we also obtain convergence in all 
intermediate topologies. The same type of argument also yields
Lipschitz dependence of the solutions on the initial data in the weaker topology, 
and in particular  uniqueness.

\bigskip

{\bf 9.   Frequency envelopes and continuous dependence.  }
Frequency envelope bounds are only needed at high frequency. They are derived 
from similar frequency envelope bounds for the paradifferential equation and 
allow us to (i) propagate higher regularity and (ii) prove  continuous dependence on the
initial data in the strong topology. 
\bigskip

\section{Multilinear and nonlinear estimates in $\lX$ type spaces}
\label{sec:mult}

In this section we recall the main bilinear and nonlinear estimates
from \cite{MMT3}, and add several related bounds that can be derived from them.

\subsection{Bilinear and Moser estimates}
Our aim here is to recall some of the estimates in \cite{MMT3},
as well as to provide some improvements adapted to the context of this paper.
We begin with bounds in $\lX^s$ spaces, where we first recall the dyadic bilinear bounds from \cite{MMT3}. These are contained within the proof of Proposition~3.2 there:

\begin{lemma}
The following bilinear estimates hold in $l^1_j X_j$ spaces:

a) High-low interactions $j < k-4$:
\begin{equation}\label{XX-hl}
\| S_j u S_k v\|_{l^1_k X_k} \lesssim 2^{\frac{jd}{2}} \|u \|_{l^1_j X_j} 
\| v\|_{l^1_k X_k} ,
\end{equation}

b) Balanced interactions, $|j-k| \leq 4$:
\begin{equation}\label{XX-hh}
\| S_i(S_j u S_k v)\|_{l^1_i X_i} \lesssim 2^{dk} 2^{-\frac{di}{2}} \|u \|_{l^1_j X_j} 
\| v\|_{l^1_k X_k}. 
\end{equation}
\end{lemma}

We note that in case (a) it suffices in effect to bound the low frequency factor in $L^\infty$, as the $l^1_k X_k$ norms depend only on the pointwise size of functions.
For the $L^\infty$ norm, on the other hand, we have the Bernstein-type inequality
\begin{equation}\label{X-infty}
\| S_j u \|_{L^\infty L^\infty} \lesssim 2^{\frac{jd}2} \| u \|_{l^1_j X_j}.  
\end{equation}

We continue with a refinement of \cite[Prop. 3.1(a)]{MMT3}:
\begin{prop}
\label{p:XX+}
a)  Let $s > \frac{d}2$. Then the $\lX^s$ spaces satisfy the
bilinear estimates
 \begin{equation}\label{u_squared}
\| u v\|_{\lX^\sigma} \lesssim \|u\|_{\lX^\sigma} 
\|v\|_{\lX^s} + \|u\|_{\lX^s} 
\|v\|_{\lX^\sigma}, \qquad d-s \leq \sigma,
 \end{equation}
 respectively
 \begin{equation}\label{u_squared+}
\| u v\|_{\lX^\sigma} \lesssim \|u\|_{\lX^\sigma} \|v\|_{\lX^s}, \qquad d-s \leq \sigma \leq s.
  \end{equation}

b) For all smooth $F$ with $F(0)= F'(0) = F''(0)=0$ we have the Moser-type 
estimate
\begin{equation}\label{moser}
\| F(u)\|_{\lX^\sigma} \lesssim   \|u\|_{\lX^\sigma} \|u\|_{\lX^s}^2 c(\|u\|_{L^\infty}), \qquad d-s \leq \sigma, 
  \end{equation}
as well as the 
difference  estimates
  \begin{equation}\label{moser+}
\| F(u) - F(v)\|_{\lX^\sigma} \lesssim   \|u-v\|_{\lX^\sigma}     (\|u\|_{\lX^s}+  \|v\|_{\lX^s})^2 c(\|u\|_{L^\infty}, \| v \|_{L^\infty}), 
\quad d-s \leq \sigma \leq s.
\end{equation}
\end{prop}

Here in (b) we are assuming that $F(u)$ is cubic in $u$ near $u = 0$ just for 
convenience, as the linear part is uninteresting and the quadratic part 
is dealt with in \eqref{u_squared} and \eqref{u_squared+}.

\begin{proof}
The bilinear estimates \eqref{u_squared} and \eqref{u_squared+} follow almost directly from \cite[(3.1)]{MMT3}.  For the high-high frequency interactions, using \eqref{XX-hh} we get
\begin{equation}\label{partAref}
\Bigl\|\sum_{\substack{i,j\ge k-4\\ |i-j|\le 4}} S_k(S_i u  S_j v) \Bigr\|_{\lX^\sigma} \lesssim  \sum_{i\ge k-4} 2^{(d-s - \sigma) i + k (\sigma-\frac{d}{2})}  \| u\|_{l^1 X^\sigma}  \|  v \|_{l^1 X^s}.
\end{equation}
For these interactions, the bound \eqref{u_squared+} follows provided $\sigma\ge d-s$.  And hence \eqref{u_squared} follows trivially.
For the low-high frequency interactions, by \eqref{XX-hl} we have
\[\Bigl\|\sum_{\substack{j<k-4\\|i-k|\le 4}} S_k (S_i u S_j v)\Bigr\|_{l^1 X^\sigma} \lesssim \sum_{j<k-4} 2^{(\frac{d}{2}-s)j} \sum_{|i-k|<4}\|S_i u\|_{l^1X^\sigma} \|v\|_{l^1 X^s}.
\]
In order to obtain \eqref{u_squared+}, and hence \eqref{u_squared}, for these interactions, we only require $s>\frac{d}{2}$.  In the remaining case, from \eqref{XX-hl} we instead obtain
\[ \Bigl\|\sum_{\substack{i<k-4\\|j-k|\le 4}} S_k (S_i u S_j v)\Bigr\|_{l^1 X^\sigma} \lesssim \sum_{i<k-4}
2^{(\frac{d}{2}-\sigma)i} 2^{k(\sigma-s)} \|u\|_{l^1X^\sigma} \sum_{|j-k|<4} \|S_j v\|_{l^1X^s}. 
\]
Here, when $\sigma\le \frac{d}{2}$, again only require $s>\frac{d}{2}$.  When, however, $\sigma > \frac{d}{2}$, in order to obtain \eqref{u_squared+}, we additionally require that $\sigma\le s$.  This in turn justifies the need for the symmetric term in \eqref{u_squared} when such an upper bound on the range of $\sigma$ is not assumed.
In \cite{MMT3}, the corresponding
 estimates were phrased in a more precise way using the concept of frequency envelopes, which one could also do here.

The Moser-type estimate \eqref{moser+} is a refinement of \cite[(3.2)]{MMT3}, which showed that if $s>\frac{d}{2}$ and $F$ is smooth with $F(0)=0$, then
\begin{equation}\label{oldMoser}
\|F(u)\|_{l^1X^s} \lesssim \|u\|_{l^1X^s} (1+\|u\|_{l^1 X^s})c(\|u\|_{L^\infty}).
\end{equation}
A trivial improvement of this is obtained in an identical manner if we eliminate
the linear part of $F$, and assume instead that $F(0) = F'(0)=0$:
\begin{equation}\label{oldMoser+}
\|F(u)\|_{l^1X^s} \lesssim \|u\|^2_{l^1X^s} c(\|u\|_{L^\infty}).
\end{equation}

To prove \eqref{moser+}, we write
\[F(u)-F(v) = (v-u) \int_0^1 F'((1-t)u+tv)\,dt,\]
so that \eqref{u_squared+} gives
\[\|F(u)-F(v)\|_{\ell^1X^\sigma} \lesssim \|u-v\|_{\ell^1 X^\sigma}
\sup_t \|F'((1-t)u+tv)\|_{\ell^1 X^s}.\]
Using \eqref{oldMoser+}, this is in turn bounded by
\[\|u-v\|_{\ell^1 X^\sigma} \Bigl(\|u\|_{\ell^1 X^s} + \|v\|_{\ell^1
  X^s}\Bigr)^2 c(\|u\|_{L^\infty}, \|v\|_{L^\infty})\]
  as desired.
  
  For \eqref{moser}, this follows directly from \eqref{moser+} with $v=0$ when $d-s \leq \sigma \leq s$.  However, we require this estimate for all $d-s \leq \sigma$; this will be critical later on  in order to achieve higher regularity bounds
  for our nonlinear evolution.  In order to prove \eqref{moser} for $\sigma > s$, we must revisit the proof of \eqref{oldMoser} from Proposition $3.1$ in \cite{MMT3}.  
  We use the expansion
  \[
  S_k F(u) = S_k F(u_0) + \int_0^\infty S_k ( u_{j} F' (u_{< j})) d j.
  \]
  Here we have temporarily replaced the discrete Littlewood-Paley composition by a continuous one with $\text{Id} = S_0+\int_0^\infty S_j\,dj$.

  For the first term we can simply use \eqref{moser} in the already studied case
  $s=\sigma$ since the norms are equivalent on $u_0$.  Indeed, we have
  \[\|S_k F(u_0)\|_{l^1 X^\sigma}\lesssim 2^{k(\sigma-s)} \|S_k F(u_0)\|_{l^1 X^s} \lesssim 2^{k(\sigma-s)} 2^{-Nk} \|S_k \partial^N (F(u_0))\|_{l^1X^s}.\]
  Upon choosing $N\ge \sigma-s$ and computing the derivative, the bound for $S_kF(u_0)$  follows immediately from \eqref{u_squared+} (with $\sigma = s$) except for the terms 
  \[\|S_k((\partial^{N}u_0) F'(u_0))\|_{l^1X^s} + \|S_k(
  (\partial^{\le N-1} u_0^2) F''(u_0))\|_{l^1 X^s}.\]
  The bounds for these terms instead follow from \eqref{u_squared+} followed by an application of \eqref{oldMoser+} and \eqref{oldMoser} respectively.

  For the integrand, as an intermediate step 
  we need to estimate the expression $f_j =F' (u_{< j}) $. A direct application 
  of \eqref{oldMoser+} yields the bound
  \begin{equation}\label{fjMoser}
  \| f_j\|_{\lX^s} \lesssim \|u\|_{\lX^s}^2 c(\|u\|_{L^\infty}).
  \end{equation}
  Differentiating any number of times and then applying \eqref{oldMoser+} yields the better high frequency bound
\begin{equation}\label{fjMoser+}
\|f_j\|_{\lX^{s+N}} \lesssim \|\partial^N f_j\|_{\lX^s} \lesssim 2^{Nj} \|u\|_{\lX^s}^2 c(\|u\|_{L^\infty}), \qquad N \geq 0,
\end{equation}
where we have argued as in the $F(u_0)$ case above to obtain the last inequality.

Then the desired estimate \eqref{moser+} is obtained using the Littlewood-Paley 
trichotomy as follows:

\smallskip

a) For high-low interactions we must have $|j-k| \leq 4$, and we bound
\[
\| S_k(u_j S_{<j-4} f_j)\|_{\lX^\sigma} \lesssim \| u_j\|_{\lX^\sigma} \|f_j\|_{L^\infty}.
\]
A subsequent application of Bernstein's inequality and \eqref{fjMoser} yields the desired estimate.

\smallskip

b) For high-high interactions we must have $j > k-4$, and we use \eqref{XX-hh}
to bound
\[
\| S_k( u_j S_{[j-4,j+4]} f_j)\|_{\lX^\sigma} \lesssim 
\ 2^{(\sigma+s-d) (k-j)} 2^{(\frac{d}{2}-s)k} \| u_j\|_{\lX^\sigma} \|f_j\|_{\lX^s}.
\]
The rapid decay with respect to both $k$ and $j$ provides the summability, and an application of \eqref{fjMoser} completes the estimate. 

c) For low-high interactions we must have $j < k-4$, and we use \eqref{XX-hl}
to bound
\[
\|  u_j P_k f_j \|_{\lX^\sigma} \lesssim 
 \ 2^{(\sigma- s-N) (k-j)} 2^{(\frac{d}{2}-s)j} \| u_j\|_{\lX^\sigma} 2^{-Nj} 
\|P_k f_j\|_{\lX^{s+N}}.
\]
For $N>\sigma-s$, we obtain sufficient decay with respect to both $j$ and $k$ to obtain summability so that \eqref{fjMoser+} completes the estimate.

This concludes the  proof of the proposition.
\end{proof}

We next turn our attention to bilinear estimates from $X\times X$ into $Y$-type spaces. We first recall the two main dyadic estimates from there:
\begin{lemma}
The following bilinear estimates hold in $l^1_j Y_j$ spaces:

a) High-low interactions $j < k-4$:
\begin{equation}\label{XXY-hl}
\| S_j u S_k v\|_{l^1_k Y_k} \lesssim 2^{\frac{jd}{2}} 2^{j-k} \|S_j u \|_{l^1_j X_j} 
\|S_k v\|_{l^1_k X_k}, 
\end{equation}

b) Balanced interactions, $|j-k| \leq 4$:
\begin{equation}\label{XXY-hh}
\| S_i(S_j u S_k v)\|_{l^1_i Y_i} \lesssim  2^{\frac{kd}{2}} \|S_j u \|_{l^1_j X_j} 
\| S_k v\|_{l^1_k X_k} .
\end{equation}

\end{lemma}

These estimates are the main building blocks for the proof of \cite[Proposition 3.1 (b)]{MMT3}. The results there are no longer sufficient in the present paper, 
where, in order to deal with the large data problem, we also need bilinear bounds where a gain is obtained when $T\ll 1$. Our result is as follows:

\begin{prop}\label{p-XXY+}
 Let $s >\frac{d}2+2$ and $T \leq 1$. Then for any $\delta \ge 0$ sufficiently small,
the following bilinear bounds hold
on the time interval $[0,T]$: 
 \begin{equation}\label{xxy2+}
\| u v\|_{\lY^{\sigma-\delta}} \lesssim  T^\delta \|u\|_{\lX^{\sigma-1}} \|v\|_{\lX^{s-1}}, \qquad 
\ \ 0 \leq \sigma \leq s,
  \end{equation}
 \begin{equation}\label{xxy1+}
\| u v\|_{\lY^{\sigma-\delta}} \lesssim  T^\delta  \|u\|_{\lX^{\sigma}} \|v\|_{\lX^{s-2}}, \qquad 
\ \ 0 \leq \sigma \leq s-1.
  \end{equation}
\end{prop}

We remark that setting $\delta = 0$ one recovers \cite[Proposition 3.1 (b)]{MMT3}.

\begin{proof}[Proof of Proposition \ref{p-XXY+}]
The main step of the proof is to establish the extensions of the bounds 
\eqref{XXY-hl} and \eqref{XXY-hh} where we gain a power of $T$. In the case
of \eqref{XXY-hl} we will trade the balance of derivatives for the $T^\delta$
gain and will prove:  
\begin{equation}\label{XXY-hl+}
\| S_j u S_k v\|_{l^1_k Y_k} \lesssim T^\delta 2^{\frac{jd}{2}} 2^{(1-\delta)(j-k)} \|S_j u \|_{l^1_j X_j} \| S_k v\|_{l^1_k X_k}, \qquad j < k-4, \ \ \delta \in [0,1].
\end{equation}
On the other hand, as it turns out, the bound \eqref{XXY-hh} already has the 
$T$ gain built in, so we will prove that
\begin{equation}\label{XXY-hh+}
\| S_i(S_j u S_k v)\|_{l^1_i Y_i} \lesssim T 2^{\frac{kd}{2}} \|S_j u \|_{l^1_j X_j} 
\|S_k v\|_{l^1_k X_k}.
\end{equation}
Once these dyadic bounds are proved, the conclusion of the proposition 
easily follows after applying the Littlewood-Paley trichotomy 
and appropriate dyadic summation.

We first prove \eqref{XXY-hl+}. Due to \eqref{XXY-hl}, we already know it for $\delta = 0$, so by interpolation it suffices to prove it when $\delta = 1$. Using H\"older's inequality 
in time we have
\[
\| S_j u S_k v\|_{l^1_k Y_k} \lesssim \| S_j u S_k v\|_{l^1_k L^1 L^2}
\lesssim \| S_j u \|_{L^\infty L^\infty} T \| S_k v\|_{l^1_k L^\infty L^2}
\lesssim T 2^\frac{jd}2 \| S_j u \|_{l^1_j X_j}  \| S_k v\|_{l^1_k X_k}
\]
as needed.

The bound \eqref{XXY-hh+} follows by a similar repeated application of 
H\"older's inequality and an application of Bernstein's inequality:
\[
\begin{split}
\| S_i(S_j u S_k v)\|_{l^1_i Y_i} \lesssim & \  \| S_i(S_j u S_k v)\|_{l^1_i L^1 L^2}
\lesssim T^\frac12 \| S_i(S_j u S_k v)\|_{l^1_i L^2 L^2}
\\
\lesssim & \ T^\frac12 2^{\frac{d}2(k-i)} \| S_i(S_j u S_k v)\|_{l^1_k L^2 L^2}
\lesssim T 2^{\frac{d}2(k-i)} \| S_i(S_j u S_k v)\|_{l^1_k L^\infty L^2}
\\
\lesssim & \ T 2^{\frac{d}2 k} \| S_j u S_k v\|_{l^1_k L^\infty L^1} \lesssim T 2^{\frac{d}2 k} 
\|S_j u \|_{l^1_j X_j} 
\|S_k v\|_{l^1_k X_k} .
\end{split}
\]
\end{proof}

\subsection{ The paradifferential source term}
\label{sec:G} \
Our goal here is to obtain estimates for the paradifferential remainder term $G(u)$.  We recall here that
\begin{equation*}
G  (u, \bar u , \nabla u , \nabla \bar{u}) = F(u, \bar u , \nabla u , \nabla \bar{u}) - \partial_j (g^{jk}  - T^w_{g^{jk}}) \partial_k u + T^w_{b_j} \partial_j u + T^w_{\tilde b_j} \partial_j \bar u.
\end{equation*}

\begin{prop}\label{p:G}
Assume that $s_0 >\frac{d}2 + 2$. Then the nonlinearity $G$ satisfies uniform bounds,
\begin{equation}\label{G-hi-sigma}
\| G(u) \|_{\lY^{\sigma}} \lesssim T^{\delta} C(\|u\|_{\lX^{s_0}}) \| u\|_{\lX^\sigma}, \quad  0 \leq \sigma.
\end{equation}
as well as Lipschitz bounds,
\begin{equation}\label{G-lo-sigma}
\| G(u_1) - G(u_2) \|_{\lY^{\sigma}} \lesssim T^{\delta} C(\|u_{1,2}\|_{\lX^{s_0}}) \| u_1-u_2\|_{\lX^\sigma}
\quad 0 \leq \sigma \leq s_0.
\end{equation}
\end{prop}

\begin{rem}
We note that in addition to the $T^\delta$ gain, here 
one could also obtain a slight gain in regularity for $G$,
which is common in quasilinear problems for the paradifferential
remainders. We do not pursue this here. By the same token,
the Lipschitz bound extends to a slightly larger range for $\sigma$. 
\end{rem}

\begin{rem}
Another refinement of the bounds \eqref{G-hi-sigma}
and \eqref{G-lo-sigma} is in the description of the 
constant $C$. Above, $C$ is allowed to fully depend on 
the $\lX^{s_0}$ norm. One could refine this and restrict
the dependence on the full $\lX^{s_0}$ norm to at most 
quadratic, but with a full dependence on the $L^\infty$ norm,
akin to Proposition~\ref{p:XX+}. Such an improvement is 
implicit in the proof below, but not needed for our results.
\end{rem}

\begin{proof}
Rewriting the equation \eqref{eqn:quasiquad1} in nondivergence form,
\[
i u_t + g^{jk}(u,\bar u) \partial_j \partial_k u = \tF(u,\bar u,\nabla u,\nabla \bar u),
\]
where 
\[
\tF(u,\bar u,\nabla u,\nabla \bar u) = 
F(u,\bar u,\nabla u,\nabla \bar u) - \partial_u g^{jk}(u,\bar u) \partial_j  u \, \partial_k u
- \partial_{\bar u} g^{jk}(u,\bar u) \partial_j  \bar{u} \, \partial_k u,
\]
one can verify by a direct computation that 
we can split $G = G_1+ G_2$ where
\[
G_1 =  (g^{jk}  - T^w_{g^{jk}}) \partial_j \partial_k u ,
\]
\[
G_2 =  \tF - T^w_{b_j} \partial_j u - T^w_{\tb_j} \partial_j \bar u,
\]
where by a slight abuse of notation we have redefined
\[
b_j = \partial_{\partial_j u} \tF, \qquad \tb_j = \partial_{\partial_j \bar u} \tF.
\]

We will work with $G_1$ directly in the form above, but $G_2$ needs
some further processing. Precisely, using a continuous Littlewood-Paley truncation indexed by the dyadic parameter $k$, we write $G_2$ in the form
\begin{equation}\label{G2}
G_2 =\tilde F(u, u_{<0}) +  \int_0^\infty  (b(u,\nabla u_{<k}) -  T^w_{b(u,\nabla u)})  \nabla u_k  + (\tilde b(u,\nabla u_{<k}) -  
T^w_{\tilde b(u,\nabla u)})  \nabla \bar u_k \, dk.
\end{equation}
The two integrated terms are similar and can be estimated
separately so we consider the first one.  There we re-expand to rewrite the integrand as
\begin{equation}\label{sec-int}
S_{>k-4} b(u,\nabla u_{<k}) \nabla u_k + \int_{k}^\infty S_{<k-4} ( b(u,\nabla u_{<j})
\nabla u_j ) \nabla u_k dj.
\end{equation}
At this point we divide the argument into two cases, depending on how
$\sigma$ compares with $s_0$.
\bigskip

{\bf Case 1, $\sigma \geq s_0$.} Here we only need to prove \eqref{G-hi-sigma}. For $G_1$ we apply the Moser estimate \eqref{moser} 
to get
\[
\|g(u)-I \|_{\lX^{\sigma}} \lesssim_M \| u\|_{\lX^{\sigma}} 
\]
and then apply Proposition \ref{p-XXY+} (specifically \eqref{XXY-hh+}).

For the first term of $G_2$ in \eqref{G2} we simply apply the Moser estimates \eqref{moser} 
to place it in $\lX^\sigma \subset T \lY^\sigma$.

Next we consider the two terms in \eqref{sec-int}. 
For the expression $b_k:=b(u, \nabla u_{<k})$ we use the Moser estimate \eqref{moser} to bound
\begin{equation}\label{b-est}
\| b_k \|_{\ell^1 X^{\sigma-1+h}} \lesssim 2^{hk} C(\|u\|_{L^\infty})
\| u\|_{l^1 X^{\sigma}}(1 + \|u\|_{l^1 X^{s_0}}), \qquad h \in [0,1],\quad s_0 \leq \sigma. 
\end{equation}
Here the case $h = 0$ follows by applying \eqref{moser} directly to $b(u, \nabla u_{<k})$ , whereas the case $h = 1$ is obtained by applying \eqref{moser} to  its gradient.
In the first term of \eqref{sec-int} we can conclude now via \eqref{xxy2+} (or, more precisely, the dyadic bounds
\eqref{XXY-hl+} and \eqref{XXY-hh+}).

 In the integrand in the second term of \eqref{sec-int} we can freely insert a projector  and rewrite it as
\[
S_{<k-4} ( S_{j} b_j \nabla u_j ) \nabla u_k, 
\]
after which we estimate applying \eqref{XXY-hl+} and \eqref{u_squared+} sequentially
\[
\begin{split}
\| S_{<k-4} ( S_{j} b_j \nabla u_j ) \nabla u_k\|_{\lY^\sigma} 
\lesssim & \ T^\delta 2^{\delta k} \| S_{j} b_j \nabla u_j\|_{\lX^{s_0-1}}   \| u_k\|_{\lX^\sigma}
\\
\lesssim & \ T^\delta 2^{\delta (k-j)} \| b_j\|_{\lX^{s_0-1}}  \|\nabla u_j\|_{\lX^{s_0-1}}   
\| u_k\|_{\lX^\sigma}.
\end{split}
\]
Now we can conclude by \eqref{b-est} with $\sigma = s_0$ and $h = 0$.

\bigskip

{\bf Case 2, $0 \leq \sigma \leq s_0$.} Here we only need to prove \eqref{G-lo-sigma}, as \eqref{G-hi-sigma} is then a straightforward consequence. We denote $v= u_2-u_1$ and use $u$ for either $u_1$ or $u_2$.

For $G_1$ we represent
\[
g(u_1) - g(u_2) = v h(u)
\]
for a smooth function $h$. Then
\[
G_1(u_1) - G_1(u_2) = \sum_k S_{>k-4}( v h(u)) \partial^2 u_k +
 S_{> k-4}(g(u)) \partial^2 v_k.
\]
The terms $h(u)-h(0)$ and $g(u)-g(0)$ are estimated in $\lX^{s_0}$ by the Moser estimate \eqref{moser}. The second term as well as the 
contribution of constants in $h$ are estimated by \eqref{XXY-hl+}, \eqref{XXY-hh+}.
One might want to also use bilinear estimates to fully include $w = h(u)-h(0)$ in $v$, but this is not allowed by the limited range of $\sigma$ in \eqref{u_squared+}. Nevertheless, we can estimate the high-low interactions in $vw$ in this manner. For the remaining terms we need to treat this as a trilinear bound,
\begin{equation}
\| S_j v S_{j} w S_k u\|_{l^1 Y^\sigma} \lesssim T^\delta 2^{- 2j -\delta j}
\| v \|_{\lX^\sigma} \| w \|_{\lX^{s_0}} \| S_k u\|_{\lX^{s_0}},
\qquad j \geq k-5.
\end{equation}
For this we separate into two cases. If $j > k+4$ we use the algebra 
property to absorb $S_k u$ into $w$, and then apply \eqref{xxy2+}.
Else, $|j-k| < 4$ so we can place the product $S_j w S_k u$
into $\lX^{s_0-1}$ with a $2^{-\delta k}$ gain, and then use \eqref{xxy2+}.

We now turn our attention to $G_2$. The contributions arising from
the first term in \eqref{G2} are easily estimated using the Moser estimates in \eqref{moser} and \eqref{xxy2+}. We next consider 
the differences corresponding to the first term in \eqref{sec-int}. Using again the notation $b_k$ for expressions of the form $b(u,\nabla u_{<k})$, the differences will be linear 
combinations of expressions of the form
\[
S_{> k-4}( v b_k) \nabla u_k, \qquad
S_{> k-4}( \nabla v_{<k} b_k) \nabla u_k, \qquad
S_{> k-4} b_k \nabla v_k.
\]
Given the bound \eqref{b-est} for $b_k$, the first term is like in the case of $G_1$ but better. The second term is worst in the case
$\sigma = s_0$, which was already covered in Case 1.
Finally the last term is as in Case 1, and can be handled using \eqref{b-est}
and \eqref{xxy2+}. 

Next we consider the integrand in \eqref{sec-int}. For any expression   $b_j = b(u,\nabla u_{\leq j})$ we use \eqref{b-est}, which allows us 
to reduce the problem to estimating  quadrilinear terms as follows:
\[
S_{<k-4} ( v b_j
\nabla u_j ) \nabla u_k \quad S_{<k-4} ( b_j \nabla v_{<j}
\nabla u_j ) \nabla u_k, \quad S_{<k-4} ( b_j
\nabla v_j ) \nabla u_k, \quad S_{<k-4} ( b_j
\nabla u_j ) \nabla v_k.
\]
 If the frequency of $v$ is $\leq k+4$ then this is no different from Case 1.  In particular, we can fully discard the last term.
Else, the worst case is $\sigma = 0$, which we assume from here on.
For the first two terms we can use the algebra property to bound 
$b_j \nabla u_j$ in $2^{-\delta j} \lX^{s_0-1}$, where the gain is 
only needed to insure the $j$ summation. Then in all three cases we are left with a trilinear bound
\begin{equation} \label{tri}
\| S_{<k-4} (v_j w_j) u_k \|_{\lY^0} \lesssim T^\delta \| v\|_{\lX^{-1}}
\| w\|_{\lX^{s_0-1}} \| u\|_{\lX^{s_0-1}}, \qquad j > k+5.
\end{equation}
We discard the multiplier, use the algebra property for $w_j u_k$
and then apply \eqref{xxy2+}.
\end{proof}

As a corollary of this, we also obtain frequency envelope bounds for $G$:

\begin{cor}\label{c:G}
Assume that $s_0 >\frac{d}2 + 2$. Let $c_k$ be a frequency envelope for $u$ in $\ell^1 X^{s_0}$.
Then, provided $\delta\ge 0$ is sufficiently small,
\begin{equation}
\| S_k G(u) \|_{\lX^{s_0}} \lesssim_M T^\delta c_k.
\end{equation}
\end{cor}

\begin{proof}
Expand 
\[
S_k G(u) = S_k G(u_{<0})+\sum_{j = 0}^\infty  S_k\Bigl[G(u_{< j+1}) - G(u_{<j})\Bigr].
\]
When $j\ge k$, we estimate the differences in $l^1Y^0$ using \eqref{G-lo-sigma}.  When $j<k$, we instead estimate the terms separately in $l^1 Y^N$ where $N > s_0+\sigma$, with $\sigma$ as in part (4) of the definition of a frequency envelope.
\end{proof}

\subsection{ Estimates for the cubic problem}

In the same way the proof of Theorem~\ref{thm:main1} relies primarily on quadratic estimates covered in the previous subsection, for the proof of Theorem~\ref{thm:main2} we need instead trilinear estimates. The
statements for these trilinear bounds are provided in this section; the proofs are similar and largely omitted.  See \cite{MMT4} for some related background. We note that here we only require $s>\frac{d+3}{2}$, which accounts for the improvement in regularity as compared to the quadratic case. We start with the replacement of Proposition~\ref{p:XX+} for $l^2 X$ type spaces:

\begin{prop}
\label{p:XX+l2}
a)  Let $s > \frac{d}2$. Then the $l^2 X^s$ spaces satisfy the bilinear estimate
  \begin{equation}\label{u_squared+3}
\| u v\|_{l^2 X^\sigma} \lesssim \|u\|_{l^2 X^\sigma} \|v\|_{l^2 X^s}, \qquad 0 \leq \sigma \leq s.
  \end{equation}

b) For all smooth $F$ with $F(0)= F'(0)  = F''(0)=0$ we have the Moser estimate
\begin{equation}\label{moser3}
\| F(u)\|_{l^2 X^\sigma} \lesssim   \|u\|_{l^2 X^\sigma} \|u\|_{l^2 X^s}^2 c(\|u\|_{L^\infty}), \qquad 0 \leq \sigma 
  \end{equation}
as well as the difference  estimates
  \begin{equation}\label{moser+3}
\| F(u) - F(v)\|_{l^2 X^\sigma} \lesssim   \|u-v\|_{l^2 X^\sigma}     (\|u\|_{l^2 X^s}+  \|v\|_{l^2 X^s})^2 c(\|u\|_{L^\infty}, \| v \|_{L^\infty}) \qquad 0 \leq \sigma \leq s.
  \end{equation}

\end{prop}
Compared to Proposition~\ref{p:XX+}, we note that here 
we have a larger range for $\sigma$. This is due to a corresponding improvement in the balanced bilinear interactions. Precisely, the bound \eqref{XX-hh} is replaced 
by 
\begin{equation}\label{XX-hh-l2}
\| S_i(S_j u S_k v)\|_{l^2_i X_i} \lesssim 2^{\frac{dk}2}  \|u \|_{l^2_j X_j} 
\| v\|_{l^2_k X_k}. 
\end{equation}

The bilinear $\lX \times \lX \to \lY $ bounds 
in Proposition~\ref{p-XXY+} are now replaced by 
trilinear $l^2 X \times l^2 X \times l^2 X \to l^2 Y $ 
type bounds. Again we will also  need the slight improvement where we trade regularity for a slight  gain in time.
\begin{prop}\label{p-XXXY+}
 Let $s >\frac{d+3}2$ and $T \leq 1$. Then for any $\delta \ge 0$ sufficiently small, 
the following trilinear bounds hold
on the time interval $[0,T]$: 
   \begin{equation}\label{xxxy+}
 \| u v w\|_{\ltY^{\sigma-\delta}} \lesssim  T^\delta  \|u\|_{\ltX^{\sigma-1}} \|v\|_{\ltX^{s-1}}   \|w\|_{\ltX^{s-1}}, \qquad 0 \leq \sigma \leq s,
   \end{equation}
 respectively
 \begin{equation}\label{xxxy2+}
\| u v w\|_{\ltY^{\sigma-\delta}} \lesssim  T^\delta \|u\|_{\ltX^{\sigma}} \|v\|_{\ltX^{s-2}}  \|w\|_{\ltX^{s-1}}, \qquad 
\ \ 0 \leq \sigma \leq s-1.
  \end{equation}
 \end{prop}

Just as in the quadratic case, the next proposition
is the main tool in the proof of the cubic bound for the paradifferential error term:

\begin{prop}\label{p:Gcubic}
Assume that $s_0 >\frac{d+3}2$. Then the nonlinearity $G$ satisfies uniform bounds,
\begin{equation}
\| G(u) \|_{\ltY^{\sigma}} \lesssim T^{\delta} C(  \|u\|^2_{\ell^2 X^{s_0}})\| u\|_{\ell^1 X^\sigma} , \quad \sigma \geq 0.
\end{equation}
as well as Lipschitz bounds,
\begin{equation}
\| G(u) - G(v) \|_{\ell^2 Y^{\sigma}} \lesssim T^{\delta} C(\|u\|^2_{\ltX^{s_0}}, \|v\|^2_{\ltX^{s_0}}) \| u-v\|_{\ltX^\sigma}
\quad 0 \leq \sigma \leq s_0.  
\end{equation}

\end{prop}

The proofs of Propositions \ref{p-XXXY+} and \ref{p:Gcubic} follow directly by modifying the proofs of Propositions \ref{p-XXY+} and \ref{p:G}  in the quadratic case and are left to the reader.


\section{Nontrapping metrics}
\label{nontrap}

Here we begin with $u_0 \in \ell^1 H^{s_0}$, $s_0>\frac{d}{2}+2$, and
fix $M$ so that
\begin{equation}
\| u_0 \|_{ \ell^1 H^{s_0}} \leq M.
\end{equation}
Then the associated metric $g(u_0)$ satisfies the Moser type bound
(cf. \eqref{moser})
\begin{equation}
\| g(u_0) - I \|_{ \ell^1 H^{s_0}} \lesssim  M^2
\end{equation}
with an implicit constant depending on the $L^\infty$ norm of $u_0$.  Here $I$ denotes the flat
background metric. In particular, by Sobolev embeddings
we have 
\[
\| g(u_0)\|_{C^2}  \lesssim  1+ M^2.
\]
This guarantees that the  Hamilton flow $(x,\xi) \to (x^t,\xi^t)$ given by
  \begin{equation}\label{geodesic}
( \dot x^t,  \dot \xi^t) = (a_\xi (x^t, \xi^t) , -a_x (x^t, \xi^t)), \ (x^0, \xi^0) = (x,\xi) 
  \end{equation}
  with $a(x,\xi) = g^{ij} (u_{0} (x)) \xi_i \xi_j$ is well-defined.

  In order to guarantee that the qualitative assumption that the
  metric $g(u_0)$ is nontrapping is meaningful we will show that this
  is in effect a condition about the bicharacteristic flow within a 
compact set. Furthermore, in order to prove the local
  well-posedness result with a bound from below for the lifespan of the
  solution we need to turn this assumption into a quantitative
  statement.

 We begin by selecting a ball $B:=B_R$ of radius $R \gg 1$ so that outside
  $B$ we have the smallness condition 
 \begin{equation}\label{R-def}
\| \chi_{>R/2} (g(u_0)-I) \|_{ \ell^1 H^{s_0} (\R^d)} \leq \epsilon
\end{equation} 
with a universal small constant $\epsilon$. Here $\chi_{>R/2}$ is a smooth cutoff which
equals $1$ outside $B_{R/2}$ and zero inside $B_{R/4}$.  We observe that for large frequencies such
a bound holds globally, so only small frequencies contribute to
this. 

 However, $R$ could still be arbitrarily large, independently of
$M$, so we retain it as one of the main parameters in our
problem. Without any restriction in generality we will make the
assumption
\begin{equation}
\log R \gg \log M 
\end{equation}
 with a universal implicit constant. The smallness outside $B$ ensures that no trapping 
can happen there. Precisely, we have the following:

\begin{lemma}\label{l:out}
  Let $B$ be as in \eqref{R-def} where $s_0>\frac{d}{2}+2$ and $\epsilon>0$ is sufficiently small.  Then any geodesic for the 
metric $g=g(u_0)$  that exits the ball $B$ will escape to spatial infinity without reentering $B/2$.
\end{lemma}

\begin{proof}
The result will follow by proving uniform bounds similar to those in \cite{MMT1}.  
We begin with the Hamilton flow for $g=g(u_0)$,
\begin{equation}
\label{hf1_g}
\dot x_i^t = g^{ij}(x^t) \xi_j^t, \quad
 \dot \xi_i^t = -(g^{jk})_i(x^t)  \xi_j^t \xi_k^t,    
\end{equation}
which we will compare with  that driven by the flat flow,
\begin{equation}
\label{hf1_flat}
\dot{\tilde{x}}_i^t =  \tilde{\xi}_i^t, \quad
\dot{\tilde{\xi}}_i^t = 0.   
\end{equation}

We will show that all outward pointing rays are forced to live in a
small angular sector. Precisely, consider a geodesic which exits
$B_R$, i.e. a solution $(x^t,\xi^t)$ to \eqref{hf1_g} starting say at $t =
0$ at the point $(x,\xi)$ so that $x \in \partial B_R$ and $\dot
x(0) \cdot x\geq 0$.  This will be compared to the corresponding
trajectory $(x + \xi t, \xi)$ for the flat flow \eqref{hf1_flat}.
We will show that the two trajectories stay close to each other,
\begin{equation}\label{flow-comp}
\| \xi^t - \xi \|_{L^\infty} < c|\xi|, \qquad \| x^t - x - t \xi \|_{L^\infty} < c t |\xi|
\end{equation}
 for $ \epsilon \ll c \ll 1$ sufficiently small. 

At the starting point  we must have 
\[
|\dot x^0 - \xi| \lesssim \epsilon |\xi|.
\]
By the continuity of $\dot x^s$, for a sufficiently small $t$, we have
\[|\dot x^s - \xi|< c |\xi|,\quad s\in [0,t].\]
It, thus, follows that \eqref{flow-comp} holds for sufficiently small
$t\ge 0$.  We then use a bootstrap argument to show that
\eqref{flow-comp} holds globally.  We assume the bounds
\eqref{flow-comp} for $t\in [0,T]$ and shall show that the same hold with, say, $c$
replaced by $c/2$.  We note two 
immediate consequences of our bootstrap assumptions.  On the time
scale $[0,T]$ for which \eqref{flow-comp} is assumed:
\begin{enumerate}[label=(\roman*)]
\item The bicharacteristic $(x^t,\xi^t)$ cannot re-enter $B/2$. 
\item The bicharacteristic $(x^t,\xi^t)$ is nearly straight, 
$|\dot x^t -  \xi | \lesssim c |\xi|$.
\end{enumerate}

It remains to complete the bootstrap.  By property (i), we will freely assume that
\[
\| g - I \|_{\ell^1 H^{s_0}} \le \epsilon
\]
in a unit neighbourhood of the bicharacteristic $(x^t,\xi^t)$, $t\in [0,T]$. By property 
(ii), the bicharacteristic remains in a cube of sidelength $2^j$ for
an $\O(2^j/|\xi|)$ amount of time.  Thus, Bernstein's inequality
yields the uniform bound
\begin{equation}\label{almost-flat}
\int_{0}^t |(g-I)(x^s)| + |\nabla_x g(x^s)| ds \lesssim 
 |\xi|^{-1} \| \chi^{ext}_{B} (g - I) \|_{\ell^1 H^{s_0}} \lesssim \epsilon |\xi|^{-1} 
\end{equation}
provided $t\in [0,T]$.

We first close the bootstrap for $\xi^t - \xi$. We have
\[
\frac{d}{dt} (\xi^t - \xi)_i = \partial_i (g^{kl})  \xi_k^t \xi_l^t.
\]
By our bootstrap assumption we have $|\xi^t| \approx |\xi|$. Hence integrating 
in the last relation and using \eqref{almost-flat} we obtain 
\[
|\xi^t - \xi| \lesssim 
\int_0^t |\partial_i (g^{kl})  \xi_k^s \xi_l^s | ds \le C \epsilon
|\xi|,\quad t\in [0,T],
\]
which suffices provided $\epsilon < \frac{c}{2C}$. 

Next we consider  the difference $x^t - x - t \xi $, for which we have
\[
\frac{d}{dt} (x^t - x - t \xi)_i = g^{ij} \xi_j^t - \xi_i.
\]
Hence,  using that $\|\chi_{>R/2}(g-I)\|_{ L^\infty} \lesssim \epsilon$, we have 
\[
| x^t - x - t \xi| \lesssim \int_0^t |g - I| |\xi^\tau|  +  | \xi - \xi^\tau| \, d \tau 
\lesssim \epsilon t |\xi|
\]
which gives the desired result with $\epsilon$ sufficiently small as
above. 
\end{proof}

Lemma \ref{l:out} ensures that any trapping must be confined to a
compact set, namely the ball $B$.  The nontrapping assumption
guarantees that all geodesics exit $B$ at both ends. We now seek
to quantify that.

In order  to measure the length of a geodesic within the ball $B$
we remark that for the geodesic flow \eqref{geodesic} we have the scaling symmetry
\[
\xi \to \lambda \xi, \qquad t \to \lambda t,
\]
where $t$ is used to denote the parameter in \eqref{geodesic} along geodesics. This one dimensional degree
of freedom needs to be removed in order to uniquely define 
the length of geodesics. Our strategy will be to project the 
geodesic flow on the cosphere bundle $\{ |\xi| = 1\}$, i.e. 
replace \eqref{geodesic} with 
 \begin{equation}\label{geodesic1}
( \dot x^t,  \dot \xi^t) = (a_\xi (x^t, \xi^t) , -a_x (x^t, \xi^t) +  (a_x (x^t, \xi^t) \cdot \xi^t) \xi^t  ), \quad |\xi^t| = 1.
  \end{equation}
Then we define the length of a geodesic $\gamma$ between two points as 
\[
\ell(\gamma) = \int_{t_0}^{t_1} |\dot x^t | dt.
\]
We remark that for a single metric $g$ it would be more natural 
to measure its length using the $g$ metric. However, here we also need to allow for changes in the metric, so it is better to have a common reference frame.

Now we are ready to define our last parameter $L$ which measures
the maximum length of a geodesic within $B$. 
Indeed, a straightforward compactness argument applied to the projection of the flow to the cosphere bundle shows that the quantity
\begin{equation}\label{Ldef}
L = \sup \{ \ell(\gamma \cap 2 B_R); \ \ \text{$\gamma$ geodesic for $g(u_0)$} \}
\end{equation}
is finite, where $\ell (\gamma)$ stands for the Euclidean length
as in \eqref{geodesic1}.  We  have an obvious lower bound  $L \gtrsim R$, but no upper bound for $L$ 
in terms of $R$. We will use $L$ as the second parameter in the quantitative  description of nontrapping.

Our next task is to see that our nontrapping assumption is stable with respect to a class of small perturbations:

\begin{prop}
\label{p:nontrap2}
  Assume that $u_0 \in \ell^1 H^{s_0}$ is so that $g(u_0)$ is
  nontrapping. Let $M,R,L$ be as above.
Then, there exists $C_0 (M) > 0$ such that for $w \in l^1 X^{s_0}$  satisfying
\begin{equation}
\|w\|_{\ell ^1 X^{s_0}} \leq e^{-C_0 (M) L}
\end{equation} 
the metrics $g(u_0+w) $  are uniformly nontrapping, with comparable parameters 
$R,L$.
\end{prop}

\begin{proof}
We denote  the two metrics by 
\[
g_0 = g(u_0), \qquad g_1 = g( u_0+w).
\]
By hypothesis we have 
\[
\| u_0 - ( u_{0}+w)\|_{\ell^1 X^{s_0}} \lesssim e^{-C_0(M) L}.
\]
Hence, for $R$ chosen as in \eqref{R-def}, outside $B_R/2$ both metrics are close to the Euclidean metric, 
\[
\|\chi^{ext}_{B/2} (g_0 - I) \|_{\ell^1 X^{s_0}} + \|\chi^{ext}_{B/2}(g_1 - I) \|_{\ell^1 X^{s_0}}\lesssim \epsilon + e^{-C_0(M) L} \lesssim \epsilon
\]
for some $\epsilon > 0$ and the analysis in Lemma~\ref{l:out} equally applies. It remains to 
compare their Hamilton flows in $B_R$, where the two metrics are close,
\[
\| g_0 - g_1\|_{C^2(B_R)} \lesssim e^{-C_0(M) L}.
\]

We begin with the Hamilton flows of both problems and show that the
trajectories of the two are close on $B_{R}$.  Namely we take
\begin{align}
\label{hf1}
& \dot x_i = g_0^{ij} (x) \xi_j, \\
& \dot \xi_i = -\partial_i g_0^{jk}(x) \xi_j \xi_k + (\xi\cdot\partial g_0^{jk}(x)\xi_j\xi_k)\xi_i  \notag
\end{align}
and
\begin{align}
\label{hf2}
& \dot {\tilde{x}}_i = g_1^{ij} (t,\tilde x) \tilde{ \xi}_j, \\
& \dot {\tilde{\xi}}_i = - \partial_i g_1^{jk}(t, \tilde x) \tilde{\xi}_j \tilde{\xi}_k + (\tilde{\xi}\cdot \partial g_1^{jk}(t,\tilde{x})\tilde{\xi}_j\tilde{\xi}_k)\tilde{\xi}_i \notag
\end{align}
with the same initial data $(x_0,\xi_0)$ at $t = 0$ where $|\xi_0|=1$, and $x_0 \in B_R$.
Our goal will be to prove
\begin{equation}
\label{hhfbd}
|x(t) - \tilde x(t)| + |\xi(t) - \tilde \xi(t)| \lesssim e^{-C(M) L} \ll 1.
\end{equation}

Rather than estimating the difference directly, it  is perhaps easiest to 
consider a one parameter family of metrics
\[
g(\cdot; h)  = (1-h) g_0 + h g_1, \qquad h \in [0,1],
\] 
define the flow $(y (t; h), \eta (t; h))$ using
\begin{align}
\label{hf1_h}
& \dot {y}_i = g^{ij} (t, y; h) \eta_j, \\
& \dot {\eta}_i = -\partial_i g^{jk}(t, y; h) \eta_j \eta_k + (\eta\cdot \partial g^{jk}(t, y; h) \eta_j\eta_k)\eta_i  \notag
\end{align}
and then differentiate in $h$.  Note that using this notation we have $(y,\eta) (t; 0) = (x(t),\xi(t))$ and $(y,\eta) (t; 1) = (\tilde x(t), \tilde \xi(t))$.  For this family we seek to prove
\begin{equation}\label{hhfbd+}
|y(t)-x(t)|+|\eta(t)-\xi(t)|\lesssim e^{-C(M)L}\ll 1
\end{equation}
uniformly in $h \in [0,1]$, which will yield \eqref{hhfbd} upon choosing $h=1$.
We shall prove \eqref{hhfbd+} using a bootstrapping argument.  We assume \eqref{hhfbd+}, and we shall prove the same with an improved constant.

The $h$ derivatives $(y_h,\eta_h)$ solve the differentiated system
\begin{align}
\label{hf1alt}
\left\{ \begin{array}{l}
 \dot  y_h  =  (g_1-g_0) (t,y) \eta  + y_h\cdot\partial (g(t,y; h) )\eta + g(t,y;h) \eta_h   ,    \\
 \dot \eta_h = -\eta \partial(g_1 -g_0) (t,y) \eta -  \eta y_h\cdot \partial \partial g(t,y;h) ) \eta - 2 \eta \partial g(t,y;h) \eta_h 
\\\qquad\qquad\qquad\qquad\qquad + (\eta \eta_h\cdot\partial g(t,y;h) \eta)\eta 
 + (\eta \eta\cdot\partial (g_1-g_0)(t,y) \eta)\eta
\\\qquad\qquad\qquad + (\eta (y_h\cdot\partial) \eta\cdot\partial g(t,y;h) \eta)\eta
 + 2(\eta \eta\cdot\partial g(t,y;h)\eta_h)\eta
 + (\eta \eta\cdot\partial g(t,y;h) \eta)\eta_h
 .
 \end{array} \right.
\end{align}
By the Mean Value Theorem, in order to establish \eqref{hhfbd+}, it suffices 
to establish the bound 
\begin{equation}
\label{hf1altv}
|y_h(t)|+ |\eta_h(t)| \lesssim e^{-2C(M) L} \ll 1
\end{equation}
uniformly for $h \in [0,1]$.

Due to the bootstrapping hypothesis \eqref{hhfbd+}, it suffices to consider \eqref{hf1alt} on $I$, which denotes the maximal time interval that a bicharacteristic for $g_0$ spends in $B_R$.  We note that $|I|\lesssim L$.  

Since the flow \eqref{hf1_h} is projected onto the cosphere bundle, we
observe that
\[\frac{d}{dt}(y_h^2+\eta_h^2) \lesssim M^{-1}e^{-2C_0(M)L}+ M (y_h^2 + \eta_h^2).\]
Using Gr\"onwall's inequality and that $|I|\lesssim L$, we obtain
\[|y_h|+|\eta_h| \lesssim M^{-1/2} L^{1/2} e^{-C_0(M)L} e^{CML}\lesssim e^{-C_0(M)L} e^{C(M)L}.\]
which for a choice of $1\ll C(M)< \frac{1}{3}C_0(M)$ proves \eqref{hf1altv} with the improved constant as desired.

We now have the uniformity of the nontrapping assumption as all trajectories exiting the ball are now sufficiently close.  Now it is clear why the $C^2$
 difference between the metrics needs to be small compared to $e^{-C_0(M)L}$. 
\end{proof}

\section{The Linear Flow}
\label{sec:LED1}

In this section we consider the $L^2$ well-posedness question for the linear 
Schr\"odinger flow
\begin{equation} \label{linear}
(i \partial_t +  \partial_k g^{kl}   \partial_l   + b\cdot \nabla)w  + \tilde b \cdot\nabla \bar{w}
= f,  \ \
w(0,x) = w_0,
\end{equation}
for large but nontrapping metrics $g$.  We also consider the corresponding 
linear paradifferential flow
\begin{equation} \label{linear-para}
(i \partial_t +  \partial_k T_{g^{kl}}   \partial_l   + T_b  \cdot\nabla) w 
+ T_{\tilde b} \cdot\nabla \bar{w}
= f,  \ \
w(0,x) = w_0.
\end{equation}
To frame the question in the context of the
previous section where nontrapping is discussed, we consider metrics $g$ and 
lower order perturbations $b$ satisfying the following properties with respect to 
parameters $M$, $R$ and $L$, which themselves are subject to 
\[
\log M \ll \log R \leq \log L:
\]

\begin{itemize}
\item Large size:
\begin{equation}
\label{gbassump}
\| g-I_d\|_{\ell^1 X^{s_0}} + \|g_t\|_{\ell^1 X^{s_0-2}} + \|(b, \tilde b)\|_{\ell^1 X^{s_0-1}} + \|(b_t, \tilde b_t)\|_{\ell^1 X^{s_0-3}} \leq M.
\end{equation}
\item
Smallness outside a ball $B_R$:
\begin{equation}
\label{smallassump}
  \| g-I_d\|_{\ell^1 X^{s_0}(B_R^c)}  + \| (b,\tilde b)\|_{\ell^1 X^{s_0-1}(B_R^c)}   
\leq \epsilon \ll 1
\end{equation}
for a fixed universal constant $\epsilon$.

\item Uniform nontrapping: For each geodesic (projected onto the cosphere bundle as in \eqref{geodesic1}) $\gamma$ at fixed time, we have 
\begin{equation}
\label{Lassump}
l(\gamma \cap 2B_R) \leq L.
\end{equation}
\end{itemize}

Our goal is to understand the energy and local energy bounds for such
a flow.  Since we will want to apply these bounds in order to solve a
nonlinear equation, it is crucial to have good control in a way that does not grow rapidly in time.  This is a
delicate matter, since the nontrapping property allows energy
growth by a $e^{C(M)L}$ factor within $B_R$ . Further, the coefficients at high frequencies may
repeatedly redirect inwards some of the outgoing energy, for additional
potential exponential growth (as we cannot assume $\epsilon$
is exponentially small with respect to $L$). 

We can prevent such iterated growth by restricting the time to a short enough 
interval:

\begin{thm}\label{t:le-L2}
Suppose that $g,b,\tilde b$ are as above, with 
$s_0 > \frac{d}{2}+2$ and associated parameters $M,R,L$.
Then the equations \eqref{linear} and \eqref{linear-para} are well-posed in $L^2$. Further,
for all small enough times $T$,
\begin{equation}\label{small-t}
T \leq e^{-C(M) L}
\end{equation}
we have a uniform bound
\begin{equation}\label{e-L2}
\| w\|_{X^0 [0,T]} \lesssim   e^{C(M) L} ( \|w_0\|_{L^2} + \| f \|_{Y^0 [0,T]} ).
\end{equation}
\end{thm}

For simplicity the above result is stated in the $L^2$ setting, but a similar result 
also holds in higher norms:

\begin{cor}
\label{lincor1}
Suppose that $g,b,\tilde b$ are as above, with $ s_0 > \frac{d}{2}+2$ and
associated parameters $M,R,L$.  Then
the equation \eqref{linear} is well-posed in $H^\sigma$ and $\ell^1
H^\sigma$ for $0 \leq \sigma \leq s_0$.  Further, for all small enough
times $T$,
\begin{equation}
T \leq e^{-C(M) L},
\end{equation}
we have a uniform bound
\begin{equation}\label{le-l1-sigma}
\| w\|_{\ell^1 X^\sigma [0,T]} \lesssim   e^{C(M) L} ( \|w_0\|_{\ell^1 H^\sigma} + \| f \|_{\ell^1 Y^\sigma} )   
\end{equation}
\begin{equation}\label{le-sigma}
\| w\|_{\ell^2 X^\sigma  [0,T]} \lesssim   e^{C(M) L} ( \|w_0\|_{H^\sigma}  + \| f \|_{\ell^2 Y^\sigma})  .
\end{equation}
The same result holds for \eqref{linear-para} for all $\sigma \geq 0$.
\end{cor}

This is not so much a corollary of the previous theorem, but rather of its proof. 
Precisely, we will be able to reuse the key part of the proof of the theorem, and then
replicate the (short) remaining parts of the proof.

Since the last result holds for all $\sigma$ in the case of the paradifferential equation, it easily implies 
the corresponding frequency envelope version, as a direct application
of the techniques in 
\cite[Propositions 5.1 and 5.3]{MMT3}.

\begin{cor} \label{fe-est}
Let $s_0> \frac{d}{2}+2$,  and suppose $w_0 \in \lH^{s_0}$, $f\in \lY^{s_0}$ with admissible frequency envelopes $\{ a_k \}$, $\{b_k\}$ respectively.  Then  
the solution $w$ to the paradifferential flow \eqref{linear-para} satisfies
\begin{equation}\label{fjbd}
\|S_k w\|_{\lX^{s_0}}  \lesssim e^{C(M) L} (a_k+b_k).
\end{equation}
\end{cor}

The essential part of Theorem~\ref{t:le-L2} is the energy estimate \eqref{e-L2}.
The $L^2$ well-posedness follows in a standard fashion from a similar 
energy estimate for the (backward) adjoint equation. Since the adjoint equation has
a similar form, with similar bounds on the coefficients, such an estimate follows
directly from \eqref{e-L2}. Thus, in what follows we focus on the proof of 
the bound \eqref{e-L2}.

An important part of the theorem is to keep good track of the dependence of the constants 
on our parameters $M$, $L$ and $T$. In order to avoid circular arguments, it is useful to apply 
a divide and conquer strategy. One step in this direction is to take $T$ out of the equation, at the expense
of allowing a lower order term in the estimate. We state this intermediate result as follows:

\begin{prop}\label{p:le-L2-lot}
Suppose that $g$, $b$, and $\tilde{b}$ are as above, with 
$s_0 > \frac{d}{2}+2$ and associated parameters $M,R,L$.
Then for any solution $w$ to \eqref{linear} or \eqref{linear-para} and all $T \leq 1$ we have a uniform bound
\begin{equation}\label{e-L2+lot}
\| w\|_{X^0  [0,T]} \lesssim   e^{C(M) L} ( \|w_0\|_{L^2} + \| f \|_{Y^0 [0,T]} + \| w\|_{L^2 L^2[0,T]} ).
\end{equation}
\end{prop}
This trivially implies Theorem~\ref{t:le-L2} by applying H\"older's inequality in time to bound the last $L^2$ norm, and
then taking $T$ small enough.

The main bound above \eqref{e-L2+lot} admits an exact counterpart in the context of Corollary~\ref{lincor1},
which is as follows:
\begin{equation}\label{le-sigma-noT}
\| w\|_{\ell^1 X^\sigma [0,T]} \lesssim  e^{C(M) L} ( \|w_0\|_{\ell^1 H^\sigma} + \| f \|_{\ell^1 Y^\sigma[0,T]}+
 \| w\|_{\ell^1 L^2 H^{\sigma}[0,T]} )   ,
\end{equation}
\begin{equation}\label{le-l1-sigma-noT}
\| w\|_{l^2 X^\sigma [0,T]} \lesssim    e^{C(M) L} ( \|w_0\|_{ H^\sigma}  + \| f \|_{l^2 Y^\sigma[0,T]}
+  \| w\|_{L^2 H^{\sigma}[0,T]} )  .
\end{equation}
Similarly, this trivially implies Corollary~\ref{lincor1} by applying H\"older's inequality in time.

The bulk of this section is devoted to the proof of this last proposition for the paradifferential equation \eqref{linear-para}.
The advantage to working with the paradifferential formulation is that it transfers 
easily to all higher regularities, whereas the paradifferential errors will only play a perturbative role.
 After that, it is much easier  to obtain the rest of the result going in reverse order.

Before presenting the proof of Proposition~\ref{p:le-L2-lot} in full detail, we first outline our strategy.
We seek to decompose the estimate into three pieces, roughly corresponding 
to a decomposition of $w$ into three components,
\[
w = \win + \wr + \wout
\]
which represent the portions of $w$ which are microlocalized near the 
incoming rays outside $B_R$, near $B_R$, respectively near the outgoing rays.
Heuristically the energy travels along bicharacteristics, and may go through all three 
stages or only the incoming and outgoing part, depending on the bicharacteristic. 
The main steps of the argument can be described as follows:

\begin{enumerate}

\item Prove an exterior incoming high frequency estimate of the form
\begin{equation*}
\| \chi_{< 100 R} \win \|_{ X^0 } \leq C \left( \| w_{0} \|_{L^2}  + \| f \|_{Y^0}\right) + C(M) \| w \|_{L^2 L^2} 
\end{equation*}
with a universal constant $C$.  This amounts to constructing a
multiplier which selects the incoming region and positive commutator
estimates. Crucially, the metric is only used in the exterior region
(outside $4B_R$), where it is a small perturbation of a flat
metric. Here the time is taken in any interval $[0,T]$ with $T \leq
1$; the smallness of $T$ is not used.

\item Using the fact that the metric  $g$ satisfies the nontrapping condition,
  we estimate the local energy inside the compact set $B_{2R}$ in terms of the incoming part,
\begin{equation*}
\| \wr \|_{ X^0} \lesssim  e^{C(M) L} \left( \| w_0 \|_{L^2}  + \| f \|_{Y^0} + 
\|\chi_{< 100 R} \win\|_{X^0} +  \| w \|_{L^2 L^2}  \right),
\end{equation*}
in a way that quantifies the potential exponential growth.

\item Rather than  controlling  the remaining component $\wout$  on the outgoing rays, we  estimate 
the full exterior part $w_{ext} = \win + \wout$
\begin{equation*}
\| w_{ext} \|_{ X^0} \lesssim   \| \wr\|_{X^0} + \| w_{0} \|_{L^2}  + \| f \|_{Y^0}  ,
\end{equation*}
simply by truncating to the exterior region and applying the small data estimate directly.

\item Combining the three estimates above we obtain
\begin{equation*}
\| w \|_{X^0} \lesssim  e^{C(M) L} \left( \| w_0 \|_{L^2}  + \| f \|_{Y^0} +  \| w \|_{L^2 L^2}  \right),
\end{equation*}
independently of the time interval size $0 < T \leq 1$.  
\end{enumerate}

\subsection{ A review of the small data results}
\label{small}
The small data problem, studied in \cite{MMT3}, provides us with a baseline for the study of the current problem.
The main assumption for the linear result there is 
\begin{equation}
\label{gbassump-small}
\| g-I_d\|_{\ell^1 X^{s}}  + \|(b, \tilde{b})\|_{\ell^1 X^{s-1}} \leq \epsilon \ll 1,
\end{equation}
where $g,b,\tilde b$ are the coefficients in \eqref{linear}.
Here the smallness of $\epsilon$ suffices to guarantee that the metric $g$ is nontrapping.
Under this assumption, we have the following counterpart of Theorem~\ref{t:le-L2}:
\begin{thm}[\cite{MMT3}, Proposition $5.2$]\label{t:le-small}
Suppose that $g$, $b$, and $\tilde b$ are as in \eqref{gbassump-small}.
Then the equation \eqref{linear} is well-posed in $L^2$. Further,
we have the uniform bounds
\begin{equation}\label{e-L2alt}
\| w\|_{X^0 [0,1]} \lesssim   \|w_0\|_{L^2} + \| f \|_{Y^0 [0,1]},
\end{equation}
\begin{equation}\label{small-l1}
\|w\|_{l^1 X^{\sigma}[0,1]}\lesssim \|w_0\|_{l^1 H^{\sigma}} + \|f\|_{l^1 Y^{\sigma}[0,1]}, \qquad 0 \leq \sigma \leq s.
\end{equation}
The same result holds for the corresponding paradifferential problem \eqref{linear-para}.
\end{thm}

\begin{rem}
The original result in \cite{MMT3} also allows for zero order terms
in \eqref{linear}, as in \eqref{eqn:quasi-lin},  with coefficients $c,\tilde c$ with regularity $c,\tilde{c} \in \ell^1 X^{s-2}$; in that case one has to limit the upper limit for $\sigma$ in \eqref{small-l1}
to $s-1$. The $X^0$ bound \eqref{e-L2alt} follows from the proof of Proposition $4.1$.  Note also that in this result the contributions of the $b$ and $\tilde b$ terms are also perturbative.
\end{rem}

\subsection{ The  $\bar w$ correction.}
\label{s:barw}

One difficulty in the large data problem, which is absent in the small data problem, is due to the $\tilde b \nabla \bar w$
term. This is perturbative in the small data case but not in the large data case. Our goal in this section is to show that
we can eliminate  this term from the paradifferential equation \eqref{linear-para} at the expense of more perturbative
terms.

Our correction will be of the form 
\begin{equation}\label{u-to-tu}
\tilde w = Sw:= w + \mathcal{R} \bar w
\end{equation}
where $\mathcal{R}$ is a paradifferential operator of order $-1$. A similar conjugation was used in \cite{Ch} to remove the complex conjugate leading order terms.  
Assuming that $w$ solves \eqref{linear-para}, we obtain the following as the equation for $\tilde{w}$:
\begin{equation}
    \label{nowbareq}
 (i\partial_t + \p_j T^w_{g^{jk}}  \p_k  + T^w_b\cdot \nabla) \tilde w = f -
 \mathcal{R}\bar f + \p_j T^w_{g^{jk}}  \p_k \mathcal{R} \bar w + \mathcal{R} \p_j T^w_{g^{jk}}  \p_k \bar w -  
T^w_{\tilde  b}\cdot \nabla \bar w + e[w],
\end{equation}
where the lower order terms $e[w]$ are given by
\begin{equation}
\label{nowbarerror}
e[w]=T^w_b\cdot\nabla\mathcal{R}\bar{w} + \mathcal{R}T^w_{\bar{b}}\cdot\nabla \bar{w} + \mathcal{R}T^w_{\bar{\tilde{b}}}\cdot\nabla w + i \mathcal R_t \bar{w}.
\end{equation}
To cancel the $\bar w$ terms on the right in \eqref{nowbareq}, we use the ellipticity of $g$.  Namely, our assumptions on $g$ ensure that we have
\[
g^{jk}  \xi_j \xi_k \geq c_0 |\xi|^2,  \ c_0 > 0.
\] 
We then select the symbol $r(x,\xi)$ of $\mathcal{R}$ to be 
\begin{equation}\label{r-def}
r(x,\xi;t) =  (1-\chi (|\xi|))   \frac{ i  \tilde{b}^j \xi_j}{2 g^{jk} \xi_j \xi_k},
\end{equation} 
where the time dependence is implicit in the $b$ and $g$ terms. Here $\chi$ is a smooth compactly supported bump function so that  $\chi(\xi) = 1$ for $\xi$ in a large neighborhood of $0$,  depending on $M$. 

In view of the regularity properties \eqref{gbassump}
for $\tilde b$ and $g$ and using the multiplicative bounds and Moser estimates in Proposition~\ref{p:XX+}, it follows that the symbol $r$
has regularity
\begin{equation}\label{r-reg}
r \in     \lX^{s_0-1} S^{-1}, \qquad r_t \in\lX^{s_0-3} S^{-1}.
\end{equation}
Here, for symbol classes we follow the notations in 
\cite{Taylor-nl}. For instance, by $ r \in \lX^{s_0-1} S^{-1}$, we mean that for each $\xi \in \RR^n$, we have $\|r(x,\xi;t)\|_{\lX^{s_0-1}} \lesssim (1+|\xi|)^{-1}$, and that each $\xi$ derivative gains one order of decay in $\xi$ with $x$ regularity that remains in $\lX^{s_0-1}$.

The paradifferential implementation of the symbol,
\begin{equation*}
r^p(x,\xi;t) = \sum_\ell S_{<\ell-4}(D_x) r(x,\xi;t) S_\ell(\xi)
\end{equation*}
is chosen such that the action of the operator only 
involves low-high interactions of the symbol  with the function on
which it is acting.   This will allow us to prove mapping properties of
$\mathcal{R} = r^p(t,x,D)$ at any regularity. 
Since without the $\chi$ cutoff $r$ is a homogeneous symbol of order $-1$, restricting to large frequencies helps to insure invertibility of the map $S w = w + \mathcal{R} \bar w$.
With this choice for $r$, we have:

\begin{lemma}
Suppose that $w$ solves \eqref{linear-para} with $g \in \ell^1 X^{s_0}$ that is uniformly elliptic, $b,\tilde b \in \ell^1 X^{s_0-1}$ as in \eqref{gbassump} and $s_0 > d/2 +2$.  Let $r$ be given by  \eqref{r-def}.
Then the transformation $S$ in \eqref{u-to-tu} is invertible\footnote{with implicit bounds depending on $M$} in $L^2$, $\ell^1 H^s$ and $\lX^s$ for $0 \leq s \leq \infty$, and  $\tilde w$ defined by \eqref{u-to-tu} solves an equation of the form
\begin{equation} \label{linear-para1}
(i \partial_t +  \partial_k T^w_{g^{kl}}   \partial_l   + T^w_b  \cdot\nabla )  \tilde w
= \tilde f,  \ \
\tilde w(0,x) = \tilde w_0,
\end{equation}
where 
\begin{equation}
\label{fY0bd} 
\| \tilde f\|_{Y^0} \lesssim_M \|f\|_{Y^0} + \| w\|_{L^2 L^2}, \qquad \| \tilde f\|_{\ell^1 Y^s} \lesssim_M \|f\|_{\ell^1 Y^s} +  \|w\|_{l^1 L^2 H^s}.
\end{equation}

\end{lemma}

This lemma allows us to reduce the proof of Proposition~\ref{p:le-L2-lot} to the case 
when $\tilde b = 0$.

\begin{proof}
We recall the regularity of the symbol $r$ in \eqref{r-reg}.

We begin with the mapping properties of $\mathcal R$ and the invertibility of the 
renormalization operator $S$ in \eqref{u-to-tu}. It suffices to show 
that $\mathcal R$ is bounded with small norm in $L^2$, $\ell^1 H^s$ and $\lX^s$.
The proof is the same for all these spaces, so to fix the notations we consider
$\lX^s$.

Since by construction we have $r \in \lX^{s_0-1} S^{-1}$, by separating variables 
it follows that we can represent the operator $\mathcal R$ as a rapidly convergent series 
of operators of the form
\[
\mathcal R = \sum_{l > l_0} 2^{-l} \sum_{m=1}^\infty r^m_{<l-4}(t,x) S^m_l(D_x).
\]
Here the subscript indicates the frequency localizations and $m$
is the summation index for separation of variables. Choosing $S^m_l$
to be uniformly bounded, we include the rapid decay in 
 $r^m \in \lX^{s_0-1}$, which  can then  be assumed to satisfy bounds of the form  
\[
\| r^m \|_{\lX^{s_0-1}} \lesssim_M m^{-N} .
\]

The dyadic multipliers $S^k_{l}$ are bounded in all of our function spaces, so in order  
to obtain bounds for $\mathcal R$ it suffices to consider bilinear multiplicative bounds,
as given in Proposition \ref{p:XX+}.  Precisely, the bound \eqref{u_squared+}
shows that 
\[
\| r^m_{<l-4}(t,x) S^m_l(D_x) u\|_{\lX^s} \lesssim \|r^m\|_{\lX^{s_0-1}} \| S_l^m u\|_{\lX^s},
\]
which after summation in $m$ and also in the dyadic index $l$ yields
\begin{equation}\label{R-small}
\| \mathcal R  u\|_{\lX^s} \lesssim_M 2^{-l_0} \|u\|_{\lX^s}.
\end{equation}
Here smallness is gained by making $l_0$ large enough, which in turn is accomplished through the choice of the cutoff $\chi$ in the definition of the symbol $r$. If we instead use the $2^{-l}$
factor to gain Sobolev regularity rather than smallness, then we obtain
\begin{equation}\label{R-Sobolev}
\| \mathcal R  w\|_{\lX^s} \lesssim_M  \| w \|_{\lX^{s-1}},
\end{equation}
as well as the related fixed time $L^2$ type bounds
\begin{equation}\label{R-Sobolev-XY}
\| \mathcal R  w\|_{H^\sigma} \lesssim_M   \| w \|_{H^{\sigma -1}}, \qquad \sigma \in \R.
\end{equation}
We similarly have the closely related fixed time bound for $\mathcal R_t$,
\begin{equation}\label{Rt-Sobolev-XY}
\| \mathcal R_t  w\|_{L^2} \lesssim_M   \|w\|_{L^2}
\end{equation}
as $r_t \in \lX^{s_0-3} S^{-1} \subset C^{-1+} S^{-1}$ (provided that $s_0 > d/2+2$).

In view of the bound \eqref{R-small}, the operator $S$ is invertible on $\lX^s$ for $0 \leq s \leq \infty$.  A similar analysis holds for $L^2$ and $\ell^1 H^s$.

Now we consider the source term bound \eqref{fY0bd}. 
From \eqref{nowbareq}, \eqref{nowbarerror} we have \eqref{linear-para1} with
\[
\tilde{f} = f - \mathcal{R}\bar f + E \bar w + e[w]
\] 
where
\begin{align*}
E = & \ \p_j T^w_{g^{jk}}  \p_k \mathcal{R}  + \mathcal{R} \p_j T^w_{g^{jk}}  \p_k -  T^w_{\tilde  b}\cdot \nabla ,
\\
e[w] = & \ 
T^w_b\cdot\nabla\mathcal{R}\bar{w} +\mathcal{R}T^w_{\bar{b}}\cdot\nabla \bar{w} + \mathcal{R}T^w_{\bar{\tilde{b}}}\cdot\nabla w + i \mathcal{R}_t \bar{w}.
\end{align*}
We need to establish the appropriate bounds for each of the terms on the right.

For $\mathcal R \bar f$ it suffices to see that $\mathcal R$ is bounded on $Y^0$, respectively $\lY^s$,
\begin{align*}
\| \mathcal R \bar{f} \|_{Y^0} \lesssim_{M} \| f \|_{Y^{0}},
\qquad \| \mathcal R \bar{f} \|_{\lY^s} \lesssim_{M} \| f \|_{\lY^{s}}.
\end{align*}
This follows in the same way as \eqref{R-small}, but using the high-low frequency interactions we have the straightforward estimate
\begin{equation}
    \label{e:XYtoY}
    \| S_{< \ell-4} u S_\ell v \|_{Y^0} \lesssim \| u \|_{L^\infty} \| S_\ell v \|_{Y^0}\lesssim \| u \|_{\lX^{s_0}} \| S_\ell v \|_{Y^0}. 
\end{equation}
instead of Proposition~\ref{p:XX+}. A similar computation applies with $Y^0$ replaced by $\lY^s$.

For the remaining terms in $\tilde f$, it suffices
to prove $L^2$ type bounds, namely
\begin{equation}\label{e-at0}
\| E \bar w\|_{L^2 L^2} + \| e[w]\|_{L^2 L^2} \lesssim_M \|w\|_{L^2 L^2},
\end{equation}
respectively
\begin{equation}\label{e-ats}
\| E \bar w\|_{l^1 L^2 H^s} + \| e[w]\|_{l^1 L^2 H^s} \lesssim_M \|w\|_{l^1 L^2 H^s}.
\end{equation}
Since all operators involved are paradifferential,
the analysis happens at fixed frequency so the two bounds are virtually identical. Thus we will focus 
on \eqref{e-at0}.

After a dyadic Littlewood-Paley localization, 
the bound for $E$ in \eqref{e-at0}
reduces to a paraproduct type 
product formula for scalar multiplications
\[
\| [r_{<l-4} g_{< l-4} - (rg)_{<l-4}] w_l \|_{L^2} \lesssim 2^{-l} \|r\|_{C^1} \|g\|_{C^1} \|w_l\|_{L^2} \lesssim  \|r\|_{\lX^{s_0-1}} \| g\|_{\lX^{s_0}} \| w_l \|_{L^2}.
\]
Here, the two terms $(rg)_{<l-4}$ cancel the high frequency contributions of $T^w_{\tilde  b}\cdot \nabla w $ and the remaining low frequency bound on $\chi T^w_{\tilde  b}\cdot \nabla w $ is easily bounded.

Finally, we consider the bound for $e[w]$ in \eqref{e-at0}.  For the $\mathcal R$ terms in $e[w]$ it suffices to use the bound in \eqref{R-Sobolev-XY},
while for the $\mathcal R_t$ term in $e[w]$ we need the bound in \eqref{Rt-Sobolev-XY}.
\end{proof}

\subsection{The incoming estimate}
\label{incoming}

To motivate the definitions that follow we briefly consider the constant coefficient case,
where $g = I_d$. Then the Hamilton flow associated to the linear constant coefficient Schr\"odinger 
equation has the form
\[
\dot x = 2 \xi.
\]
Hence rays are straight lines, which approach the origin when $x \cdot \xi < 0$, or, in other words, 
as long as the angle $\theta = \angle(x,\xi)$ satisfies $ \cos \theta < 0$.  

In the problem we are considering, the coefficients are not constant
but are small in the exterior region, and, as seen in
Section~\ref{nontrap} the direction for the bicharacteristics does not
deviate much from being constant.  Thus we can define the incoming
part of a solution $w$ using a pseudodifferential truncation,
\[
\win = P_{in}(x,D) w,
\]
where the symbol of $P_{in}$ is chosen as
\[
p_{in}(x,\xi) = \chi_{in}(\cos \theta) \chi_{>5R}(|x|) .
\]
Here $\chi_{in}$  is a nonincreasing cutoff which selects the interval $[-\infty,-\frac38)$, and $\chi_{>5R}$ 
is nondecreasing and selects the exterior of $5B_R$. 
Then the aim of this section is to establish the bound
\begin{equation}\label{w-inc}
\| \chi_{< 100 R} \win \|_{ X^0 } \lesssim  \| w_{0} \|_{L^2}  + \| f \|_{Y^0} +  C(M) \| w \|_{ L^2L^2} .
\end{equation}
We remark here that one could remove the  $\chi_{< 100 R}$ truncation, as well as the $M$ dependence in the constant in the last term. But this would require extra work, and the result is not needed in our sequence of steps.

As a first minor simplification, we can use the small data results of
\cite{MMT3} (see Theorem~\ref{t:le-small} above) to reduce the problem to the case when the portion of $f$ outside of $B_{3R}$ need only be measured in $L^2L^2$:
\begin{equation} \label{better-f}
\chi_{>3R} f \in L^2 L^2.
\end{equation}

By truncating the
coefficients $g-I_d$ and $b$ outside $B_R$,
\[
g_{ext} - I_d = \chi_{> R}(g - I_d), \qquad b_{ext} = \chi_{> R} b  \]
 we obtain coefficients $g_{ext} - I_d, b_{ext}$ which are small overall in the norms of \eqref{gbassump}
and coincide with $g-I_d, b$ in the exterior region.  Then we solve the auxiliary problem
\begin{equation} \label{paralinear-loc}
(i \partial_t + \partial_k T^w_{g_{ext}^{kl}}  \partial_l   + T^w_{b_{ext}}  \cdot\nabla) \tilde w = f,  \ \
\tilde w(0) = 0.
\end{equation}
By Theorem~\ref{t:le-small}, $\tilde w$ satisfies a global favorable estimate,
\begin{equation}\label{smallletilde}
\| \tilde w\|_{X^0} \lesssim \|f\|_{Y^0}.
\end{equation}
On the other hand, for the difference 
\[
w_1 = w - \chi_{>2R} \tilde w
\]
we have the equation
\[
(i \partial_t + \partial_k T^w_{g^{kl}}  \partial_l   + T^w_b  \cdot\nabla) w_1 = (1-\chi_{>2R}) f+ f_1 + f_2 + f_3
\]
where $f_1$ arises from the change of metric,
\[
f_1 = - ( \partial_k T_{g^{kl} - g_{ext}^{kl}}   \partial_l   + T_{b- b_{ext}}  \cdot\nabla) \chi_{>2R} \tilde w 
\]
and $f_2$, $f_3$ from the localization of $\tilde w$ to the exterior region,
\[
f_2 = (1-\chi_{>3R}) [ \chi_{>2R}, (\partial_k T_{ g_{ext}^{kl}}   \partial_l   + T_{b_{ext}}  \cdot\nabla)] (1-\chi_{> 3R}) \tilde w .
\]
\[
f_3 = \chi_{>3R} [ \chi_{>2R}, (\partial_k T_{ g_{ext}^{kl}}   \partial_l   + T_{b_{ext}}  \cdot\nabla)] (1-\chi_{> 3R}) \tilde w 
+  [ \chi_{>2R}, (\partial_k T_{ g_{ext}^{kl}}   \partial_l   + T_{b_{ext}}  \cdot\nabla)] \chi_{> 3R} \tilde w .
\]

It is easily seen that $f_1$ is a Schwartz function in $x$ since the functions $\chi_{>2R}$  and 
$g^{kl} - g_{ext}^{kl}$, $b- b_{ext}$ have separated localizations. In particular we have
\begin{equation}\label{f1}
\|f_1\|_{L^2} \lesssim \|\tilde{w}\|_{L^2L^2}.
\end{equation}
A similar bound holds for $f_3$ due to the separated
localizations of $\chi_{> 3R}$ and $\nabla \chi_{>2R}$. In particular we have
\begin{equation}\label{f3}
\|f_3\|_{L^2} \lesssim \|\tilde{w}\|_{L^2L^2}.
\end{equation}
The function $f_2$, on the other hand, is localized 
and satisfies
\begin{equation}\label{f2}\|f_2\|_{Y^0} \lesssim \|\chi_{<4R} \tilde{w}\|_{L^2H^{1/2}} 
\end{equation}
to which we may subsequently apply \eqref{smallletilde}. This completes our reduction to the case when \eqref{better-f} holds.

Now we return to the estimate for $\win$, under the additional assumption \eqref{better-f} on $f$.
The main step in the proof of this bound is based on the positive
commutator method, using a well chosen order zero formally self-adjoint
pseudodifferential multiplier $Q_{in} \in OPS^0$ with symbol  supported
in the incoming region. 

Writing the second order part of the Schr\"odinger operator in divergence form,
we will use the notations
\[
\partial_k T^w_{g^{kl}}  \partial_l + T^w_{b^j} \partial_j = P + B
\]
where the principal part $P$ is  self-adjoint.
To understand the choice of $Q_{in}$, which we take to be a self-adjoint operator of order zero that is independent of $t$, we first compute formally using
the $L^2$ inner product in $\R^d$:  
\[
\begin{split}
\frac{d}{dt}  \Re \langle Q_{in} w, w \rangle = \ & 
 i \langle  Q_{in} (P+B) w, w \rangle - i  \langle  Q_{in}w,  (P+B) w \rangle + 2 \Re  \langle Q_{in} w, - i f \rangle
\\
= & \  \langle i ([Q_{in},P] + Q_{in} B -  B^* Q_{in}) w, w \rangle  + 2 \Re  \langle Q_{in} w, - i f \rangle
\\
: = & \ - \langle C w, w \rangle    + 2 \Re  \langle Q_{in}w, - i f \rangle  .
\end{split}
\]
Here the operator $C$ is an order one 
self-adjoint pseudodifferential
operator. 
Integrating between $0$ and $T$ with $T \leq 1$ we obtain
\begin{equation}\label{com-relation}
 \Re \langle Q_{in} w, w \rangle (T) +  \int_0^T \langle C w, w \rangle dt =   \Re \langle Q_{in} w, w \rangle (0) +
2 \int_0^T \Re  \langle Q_{in}w, -i f \rangle dt .
\end{equation}
The first term on the right is easily estimated in terms of the $L^2$ norm of the data,
\[
\Re \langle Q_{in} w, w \rangle (0) \leq \| w(0)\|_{L^2}^2.
\]
For the second term on the right we can take advantage of \eqref{better-f} and use that we shall take $Q_{in}$ to be supported\footnote{ Here we harmlessly gloss over the difference between the support of the symbol and that of the kernel; this can be readily rectified
with an $OPS^{-\infty}$ adjustment to $Q_{in}$.} where $|x|>4R$ to obtain
\[
 \Re  \langle Q_{in}w, i f \rangle   \lesssim \| w\|_{L^2}
  (\| \chi_{>3R} f\|_{L^2} + \|  \chi_{<4R} f\|_{Y^0}),
 \]
where the interior $Y^0$ bound arises to account for smoothing tails.

Our goal therefore is to choose the operator $Q_{in}$ favorably so that we can prove a good bound from below
for the left hand side of \eqref{com-relation}. This is easily done for the first term, where it suffices to impose the conditions
\begin{equation}\label{q-pos}
q(x,\xi) \geq 0, \qquad q(x,\xi) \gtrsim 1 \ \ \text{ in } \{ \cos \theta <- \frac14, \ |x| > 4R \}.
\end{equation}
Then G{\aa}rding's inequality shows that 
\begin{equation}\label{QatT} 
\Re \langle Q_{in} w, w \rangle (T) \gtrsim \| \win(T)\|_{L^2}^2 - C \|w(T)\|_{H^{-\frac12}}^2,
\end{equation}
where the last term can be further estimated in a naive fashion
by the energy type relation  
\begin{equation}\label{lot-bound}
\|w(T)\|_{H^{-\frac12}}^2 \lesssim  \|w(0)\|_{H^{-\frac12}}^2  + \| f\|_{L^2H^{-1}}^2
+ C(M) \| w\|_{L^2 L^2}^2.  
\end{equation}
Indeed, \eqref{lot-bound} results from considering $2\Im\la (i\partial_t + \partial_k T^w_{g^{kl}}\partial_l + T^w_b\cdot\nabla)w, \la D\ra^{-1} w\ra$.  Note, in this estimate there will be implicit dependence upon $\|b\|_{L^\infty}$ and hence on our large constant $M$.

We next consider the time integral involving the operator $C$. For
this we will seek to carefully apply G{\aa}rding's inequality in order to
prove the bound
\begin{equation}\label{C-all}
 \int_0^T \langle C w, w \rangle dt \gtrsim \| \chi_{< 100 R} \win \|_{X^0}^2 - \| w\|_{L^2 L^2}^2 .
\end{equation}
Together with \eqref{QatT}, this would complete the proof of the desired estimate \eqref{w-inc}.
The rest of this section is concerned with the choice of $Q_{in}$ so that the above estimate holds.

Within the ball $B_{100 R}$, the $X$ norm is equivalent to the $L^2 H^{\frac12}$ norm, therefore at least heuristically 
the principal symbol $c_0(x,\xi)$ of $C$ should satisfy 
\begin{equation}
c_0(x,\xi) \geq 0, \qquad c_0(x,\xi) \gtrsim |\xi|  \text{ in } \{ \cos \theta < - \frac14, \ 100 R > |x| > 4R \}.
\end{equation}
However, the matters are a bit more delicate because of the coefficients $g$ and $b$ which have limited regularity.

We write $C$ in the form
\[
C =  i [\Delta,Q_{in}] +    ( i[(P-\Delta),Q_{in}]+i Q_{in} B - i  B^* Q_{in}) := C^{main} + C^{err}
\]
and begin with computing the principal symbol of $C$, which is 
\[
\begin{split}
c_0(x,\xi;t) 
= & -\  2 \xi  \cdot \partial_{x} q_{in}(x,\xi)
   - \{ (g^{ij}-\delta^{ij}) \xi_i \xi_j, q_{in}(x,\xi)\}    - 2 \Im  b(t,x,\xi) q_{in}(x,\xi)
\\
:= & \ c_0^{main} + c_0^{err}.
\end{split}
\] 
Note that, as $g,b$ depend upon both $x$ and $t$, $c_0$ also has some time dependence.  However, as we will only rely upon the spatial regularity of those terms for the analysis of this operator, we will drop the explicit reference to the $t$ dependence below for convenience.
We will choose $q_{in}$ so that the first term yields a bounded nonnegative  contribution
that also controls the remaining terms. 

Precisely, our assumption \eqref{smallassump} for the coefficients $g-I_d$ and $b$ 
of our Schr\"odinger operator guarantees that we have a uniform pointwise bound
\[
|g-I_d| + |\nabla g| + |b|  \lesssim \epsilon \mu^2_k, \qquad R \leq k <|x| < k+1,
\]
where the sequence $\{\mu_k\}$ is square summable, 
\[
\sum \mu_k^2 \lesssim 1, \qquad \mu_R \approx 1.
\]
Without any restriction in generality we also assume that the sequence $\mu_k$ is slowly varying. Then,
we can find a increasing  function 
\[
\rho_R :[R,\infty) \to [1,2]
\]
so that 
\[
\rho'_R(r) \gtrsim \mu_k^2, \qquad \rho_R^{(j)}(r) \lesssim_j \mu_k^2, \  j \geq 2, \qquad r \in [k,k+1].
\]
Using this weight $\rho_R$, we define the symbol $q_{in}$ as
\[
q_{in}(x,\xi) = \rho_R(r) \chi_{> 5R}(r) \chi_{in}(\cos \theta - c \rho_{R}(r)),
\]
where $\chi_{in}(\rho)$ is nonincreasing, supported in $\{\rho < -1/4\}$, and identically $1$ on $\{\rho<-1/2\}$.  The constant $c$ is small and satisfies
$\epsilon \ll c \ll 1$.
With this choice, it suffices to prove the following:

\begin{lemma}
With $q_{in}$ chosen as above, the estimate \eqref{C-all} holds.
\end{lemma}

\begin{proof}
The leading part $c_0^{main}$ of $c_0$ is given by 
\begin{multline*}
 \frac{1}{2}c_0^{main}(x,\xi) 
 = 
 - |\xi| \cos\theta \chi_{in}(\cos \theta - c \rho_{R}(r)) ( \rho_R \chi_{> 5R})'(r)\\- \chi_{in}'(\cos \theta - c \rho_{R}(r) )  \frac{|\xi|}{r} \sin^2 \theta   ( \rho_R \chi_{> 5R})(r) + c\chi_{in}'(\cos \theta - c \rho_{R}(r) ) |\xi| \cos \theta  (\rho_R \chi_{> 5R})(r) \rho_R'(r)  .
\end{multline*}
Provided $c$ is sufficiently small, all three terms are nonnegative, and we get the size
\begin{align*}
& c_0^{main}(x,\xi)  \approx  |\xi|  \chi_{in}(\cos \theta - c \rho_{R}(r)) ( \rho_R \chi_{> 5R})'(r) - \chi_{in}'(\cos \theta - c \rho_{R}(r))  \frac{|\xi|}{r}  ( \rho_R \chi_{> 5R})(r) \\
& \hspace{3cm} - c\chi_{in}'(\cos \theta - c \rho_{R}(r)) |\xi|  (\rho_R \chi_{> 5R})(r) \rho_R'(r) .
\end{align*}
On the other hand, for the remaining terms in $c_0$ 
we have a favorable bound
\begin{equation}\label{c0-err}
|c_0^{err}| \lesssim \epsilon  \left(  |\xi| \chi_{in}(\cos \theta - c \rho_{R}) ( \rho_R \chi_{> 5R})'(r) - \rho_R'(r)\chi_{in}'(\cos \theta - c \rho_{R})  |\xi|  ( \rho_R \chi_{> 5R})(r) \right).
\end{equation}

It is easier to argue in the case of $C^{main}$, which belongs to $OPS^1$. Since its principal symbol 
is nonnegative and of size $\xi$ within the region
\[
\{ 4R < |x| < 100 R, \  \cos \theta < -\frac12 \},
\]
by the classical G{\aa}rding inequality for $S^1$ symbols we have the fixed time bound
\begin{equation}\label{c-main}
\langle C^{main} w, w \rangle  = \langle (c_0^{main})^w(x,D) w, w \rangle + O(\|w\|_{L^2}^2) \gtrsim
\| \chi_{<100 R} w_{in} \|^2_{H^\frac12} - C \| w\|_{L^2}^2.
\end{equation}

The similar bound for $C^{err}$ is slightly more delicate. We will show that it
satisfies the fixed time bound 
\begin{equation}\label{c-err}
\langle C^{err} w, w \rangle \lesssim \epsilon (\langle (c^{main}_0)^w(x,D) w, w \rangle + \|w\|_{L^2}^2).
\end{equation}
If we have this, then combining the bounds \eqref{c-main} and \eqref{c-err}
we obtain
\[
\langle C w, w \rangle  \gtrsim  \langle (c_0^{main})^w(x,D) w, w \rangle - C \|w\|_{L^2}^2 \gtrsim
\| \chi_{<100 R} w_{in} \|^2_{H^\frac12} - C \| w\|_{L^2}^2,
\]
which after time integration yields \eqref{C-all} and in turn gives \eqref{w-inc}. It remains to prove 
\eqref{c-err}.

The difficulty here is that the symbol $c^{err} \in C^1\cdot S^1$
since it involves spatial coefficients depending upon $b$ and the
Poisson bracket with coefficients in $g$ (and hence involves
derivatives of  $g$), which means we only have bounds in $X^{s_0 - 1}$ and cannot guarantee enough regularity 
to allow us to directly use G{\aa}rding's inequality, see \cite{TataruSumsSquares}. Instead we will make a more careful symbol analysis.
We first reconsider $c^{main}_0$, for which we write a sum of squares decomposition, 
\[
c^{main}_0 = \phi_1^2 + \phi_2^2 + \phi_3^2,
\]
where $\phi_1, \phi_2, \phi_3 \in S^{\frac12}$ are smooth nonnegative symbols given by 
\begin{align*}
 \phi_1^2 &=  -|\xi| \cos\theta \chi_{in}(\cos \theta - c \rho_{R}(r)) ( \rho_R \chi_{> 5R})'(r),\\ 
 \phi_2^2 &= - \chi_{in}'(\cos \theta - c \rho_{R}(r))  \frac{|\xi|}{r} \sin^2 \theta   ( \rho_R \chi_{> 5R})(r) , \\
 \phi_3^2 &=   c |\xi| \cos \theta \chi_{in}'(\cos \theta - c \rho_{R}(r))  (\rho_R \chi_{> 5R})(r) \rho_R'(r).
\end{align*}
We also consider a fourth nonnegative symbol $\phi_4 \in S^\frac12$ given by 
\[
\phi_4^2 = |\xi| \rho'_R(r) q_{in}(x,\xi) .
\]

By pseudodifferential calculus we have the fixed time bound 
\[
\langle C^{main} w, w \rangle = \| \Phi_1(x,D) w\|_{L^2}^2 + \| \Phi_2(x,D) w\|_{L^2}^2  + \| \Phi_3(x,D) w\|_{L^2}^2 + O(\| w\|_{L^2}^2).
\]
On the other hand, 
\[
\phi_4^2 \lesssim \phi_1^2 + \phi_2^2 + \phi_3^2.
\] 
Therefore by G{\aa}rding's inequality we have
\[
 \| \Phi_4(x,D) w\|_{L^2}^2  \lesssim \| \Phi_1(x,D) w\|_{L^2}^2 + \| \Phi_2(x,D) w\|_{L^2}^2  + \| \Phi_3(x,D) w\|_{L^2}^2 + O(\| w\|_{L^2}^2).
\]
Now we consider the symbol $c^{err} \in C^1\cdot S^1$. Modulo an $L^2$ bounded $C^0\cdot S^0$ component 
we can replace it with its principal part $c^{err}_0$. Given its expression, it is easily seen that 
we can use the above $\phi_j$'s to represent the principal part $c^{err}_0$ in the form
\[
c^{err}_0 = \sum_{j = 1}^4 d_j(x,\xi) \phi_j^2(x,\xi), \qquad d_j \in \epsilon C^1S^0,
\]
which is a more careful substitute for \eqref{c0-err}.
Then at the operator level
we can write
\[
c^{err}_0(x,D) = \sum_{j = 1}^4  \Phi_j(x,D)^*   D_j(x,D) \Phi_j(x,D) + \epsilon OPC^0S^0,
\]
which yields the bound
\[
\langle C^{err} w, w \rangle \lesssim \sum_{j=1}^4 \| \Phi_j(x,D) w\|_{L^2}^2 + \| w \|_{L^2}^2.
\]
Thus \eqref{c-err} follows, and the proof is complete.
\end{proof}

\subsection{Nontrapping estimates on $B(0,R)$} 
\label{compact}

Here we use the nontrapping condition to produce a high frequency bound for
 $w$ within the compact set $B(0,R)$, in terms of the incoming part $\win$ estimated in the previous subsection. 
 Precisely, we will show that
\begin{equation}\label{w-comp}
\| \chi_{<R} w \|_{X^0} \lesssim  e^{C(M) L} \left( \| w_0 \|_{L^2}  + \|\chi_{<100R} f \|_{Y^0} + 
\| \chi_{<100R} \win\|_{X^0} +  \|w\|_{L^2 L^2}\right).
\end{equation}

To clarify the meaning of the norms in \eqref{w-comp}, we recall that within a compact set (e,g. $100B_R$)
the $X^0$ norm is equivalent to the $L^2_t H^\frac12$ norm, while the $Y^0$ norm is equivalent to the 
$L^2 H^{-\frac12}$  norm. 

This proof also uses a positive commutator argument, based on
propagation of singularities in $B(0,100 R)$.  The key idea is that
any bicharacteristic ray which enters $2B_R$ is coming from the
phase space support of $\chi_{<100R} \win$.  From a purely qualitative
perspective, the nontrapping condition implies that such an estimate
must hold, with a suitable implicit constant. The challenge is to
carefully track the constant.

For the positive commutator argument we use again a nonnegative
pseudodifferential multiplier $Q_{comp} \in OPS^0$, whose symbol is
this time supported in $100B(0,R)$, and repeat the computation leading
to \eqref{com-relation}. We split the analysis into an ode part, where
we construct the symbol for $Q_{comp}$, and a microlocal part, where
we use the properties of symbol in order to prove the desired estimate
\eqref{w-comp}. We begin with some heuristic considerations.

\begin{itemize}
\item The size of the frequency $\xi$ may vary considerably along the
  Hamilton flow of the operator $g^{ij} \xi_i \xi_j$. To avoid difficulties arising from this, we
  will take the symbol $q_{comp}$ to be homogeneous in $\xi$ for
  $|\xi| \gg 1$.

\item The time $t$ also varies along the Hamilton flow, which would
  seem to require a time dependent construction of $q_{comp}$.
  However, the propagation speed is proportional to the frequency size
  $|\xi|$, therefore in the high frequency limit the time is constant
  along the flow. Because of this, we will only use the fixed time
  flow in the $q$ construction.  
\end{itemize}

To construct $q_{comp}$ we will only use the principal symbol for the
Schr\"odinger operator,
\[
a(x,\xi) = g^{ij} \xi_i \xi_j
\]
and
\[
H_a = \sum_{j=1}^d \partial_{\xi_j} a \partial_{x_j} - \partial_{x_j} a \partial_{\xi_j}
\]
its Hamiltonian vector field.  We summarize our result as follows:

 \begin{prop}\label{Doi_construction_pert}
   Assume that the coefficients $g^{ij}$ satisfy the conditions in
   \eqref{gbassump}, \eqref{smallassump}.  Moreover, we assume that
   the flow \eqref{geodesic1} permits no trapped geodesics,
   with the longest within $B_{2R}$ with $|\xi| = 1$ being of length $L$.  Let $C \gg 1$ be a large
   universal constant.

Then there exists a smooth, homogeneous, real-valued, nonnegative 
symbol $ q \in S^0$ in the region $|x| \leq 100 R$ with the following properties:
\begin{enumerate}
\item Support property: 
\[
\supp q \cap \{|x|>50R\} \subset \{ \cos \theta < - \frac12 \}.
\]
\item Size:
\begin{equation}
| \partial_t^j \partial_x^\alpha \partial_\xi^\beta q(x,\xi)| \lesssim e^{CM L} |\xi|^{-|\beta|}.
\end{equation}

\item Positive commutator:
\begin{equation}
- H_{a} q  \geq CM |\xi|  q 
\end{equation}
and 
\begin{equation}
- H_{a} q  \geq |\xi|  \qquad \text{inside } \  \{  |x| < 2 R \}
\end{equation}
\item Bounded gradient,
\begin{equation}
 |q_x| + |\xi| |q_\xi|  \lesssim e^{CML} H_{a} q  .
\end{equation}
\end{enumerate}
\end{prop}

\begin{proof}
From the conditions \eqref{gbassump} we have the following regularity properties for $g$ in a compact set:
\[
\| g\|_{C^{2,\delta}}  + \| g_t\|_{C^\delta} \lesssim M \qquad \text{in } B(0,R)
\]
respectively 
\[
\| g\|_{C^{2,\delta}}  + \| g_t\|_{C^\delta} \lesssim \epsilon \qquad \text{outside } B(0,R).
\]
We will only use these properties in the proof of the proposition.

This would be a standard construction for smooth $g$; similar
constructions have already been done in \cite{StTa} as well as in
Proposition $3.6$ in \cite{MMT2}.  The difficulty we encounter in the
nonsmooth case is that a direct construction based on the Hamilton
flow of $a$ will yield a nonsmooth $q$. There are two possible
strategies here, to regularize $g$ and then construct $q$ or
vice versa; both work, but we choose the former.

Since we seek $q$ homogeneous, we will work on the cosphere bundle
$S^* \R^d = \{(x,\xi); |\xi| = 1\}$ using the notion of flow in \eqref{geodesic1}.  

We begin by regularizing $g$. Given a frequency scale $\lambda$, to be
chosen later, we regularize $g$ to $g_\lambda$ at scale $\lambda$ in
$x$ and at scale $\lambda^2$ in $t$. Then our uniform bounds above
imply that
\[
|g -g_\lambda| \lesssim \lambda^{-2}, \qquad |\nabla g - \nabla g_\lambda| \lesssim \lambda^{-1}.
\]
Next we compare their Hamilton flows $\Phi(s;x_0,\xi_0)$, respectively $\Phi_\lambda(s;x_0,\xi_0)$
starting at any point  $(x_0,\xi_0)$ with $x_0 \in 100B_R$:

\begin{lemma}
Assume that 
\[
\lambda \geq e^{2CML}, \qquad C \gg 1.
\]
Then the flows of $\tilde H_a$ and $\tilde H_{a_\lambda}$, defined as in \eqref{geodesic1}, stay close 
\[
|\Phi(s;x_0,\xi_0) - \Phi_\lambda(s;x_0,\xi_0)| \lesssim e^{-C ML}
\]
until exiting $100 B_R$.
\end{lemma}
\begin{proof}
Recall from \eqref{geodesic1}, that the flow $\tilde H_a$ can be written
 \begin{equation}\label{geodesic1_vf}
\tilde H_a = \sum_{j=1}^d \partial_{\xi_j} a \partial_{x_j} - (\partial_{x_j} a - (\nabla a \cdot \xi)  \xi_j ) \partial_{\xi_j}.
  \end{equation}
This then follows directly from a small modification of the arguments in Section \ref{nontrap}, Proposition \ref{p:nontrap2}.  
\end{proof}

In particular this shows that
for such $\lambda$ the $\tilde H_{a_\lambda}$ flow is also uniformly
nontrapping, with comparable parameters and hence we may choose $L$ such that it is a bound on the longest geodesic for both flows within $B_R$.

The first element of our construction is  a smooth $0$ homogeneous
nonnegative symbol $\chi$ in $\{ |x| < 100 R\}$ that is incoming
relative to the flat metric, i.e. $a_0 = |\xi|^2$.  This is chosen akin
to the incoming localization in the previous subsection, namely
$q_{in}$, but shifted so that its support includes $B(0,2R)$, i.e. for
instance
\[
\chi(x,\xi) = \chi_{>2R} (| x - 8R \xi|)  \chi_{<-1/2} (\cos\angle(x-8R\xi,\xi)), \qquad |\xi| =1.
\]
The three important properties of this symbol are as follows:

\begin{itemize}
\item Incoming relative to $a_0$,
\begin{equation}\label{chi-dchi}
- H_{a_0} \chi \gtrsim |\nabla \chi| \text{ in } 100B_R.
\end{equation}

\item Covers $2B_R$, 
\[
\chi \gtrsim 1 \quad \text{ in } 2B_R.
\]

\item  Narrower than the previous incoming multiplier $q_{in}$,
\[
\supp \chi \cap \{ |x| > 50 R \} \subset \{ \cos \theta < - \frac12\}.
\]
\end{itemize}

We use $\chi$ to construct  our smooth nonnegative symbol $q$,
by solving the ODE
\[
- \tilde H_{a_\lambda}  q = CM  q + \chi
\]
with initial data set by the condition 
\[
\supp  q \subset \supp \chi.
\]
In other words, we solve the ode backwards along the $\tilde H_{a_\lambda}$ flow,
beginning when $\chi$ is first encountered.  Note that here $q$ will depend upon $t$ since $a$ does.  Since $a$ is of class
$C^2$, the Hamilton flow of $\tilde H_{a_\lambda}$ is well-posed. Since the $C^2$ norm
of $a_\lambda$ is bounded by $M$ and the longest trajectory has length $L$
(within $100 B_R$), by Gr\"onwall it follows that
\[
|q| + |\nabla q| \lesssim e^{CML}
\]
for $x \in 100 B_R$.
We can also estimate higher regularity for $q$, losing only powers of $\lambda$ when differentiating
in $x,\xi$, respectively $\lambda^2$ for time derivatives. 
Thus, if $\lambda$ is chosen so that $\lambda \approx e^{CML}$, then
 $q$ already has all the properties in the Proposition \ref{Doi_construction_pert}, but relative to $a_\lambda$.
In order to switch from $a_\lambda$ to $a$ we need one more piece of information,
namely that 
\begin{equation}\label{q-dq}
- \tilde H_{a_\lambda}  q \gtrsim e^{-CML} |\nabla  q|.
\end{equation}
To prove \eqref{q-dq}, we compute the ODEs for the two quantities,
\[
- \tilde H_{a_\lambda} (\nabla  q) = CM \nabla  q + O(|\nabla^2 a|) \nabla  q + \nabla \chi
\]
whereas 
\[
- \tilde H_{a_\lambda} \tilde H_{a_\lambda}  q = CM  \tilde H_ {a_\lambda} q - \tilde H_{a_\lambda} \chi  
\]
and \eqref{q-dq} is easily seen to follow by the comparison principle for ODE's in view of \eqref{chi-dchi}.

Now we proceed to the final step of our construction, which is to
replace $a_\lambda$ by $a$.  We have
\[
| \tilde H_{a_\lambda} q - \tilde H_{a} q| \lesssim |\nabla a - \nabla a_\lambda| |\nabla q| \lesssim M \lambda^{-1} 
e^{CML} |\nabla q| \lesssim - e^{-CML}  \tilde H_{a_\lambda} q
\]
which suffices. 
\end{proof}

We define the symbol $q_{comp}$ by fine tuning the symbol $q$ constructed above in the exterior region.
First we consider a symbol $\tilde \chi$ which is akin to $\chi$ but with slightly larger support,
so that 
\[
\tilde \chi \gtrsim 1 \qquad \text{ in }  \supp q.
\]
As in the previous subsection, we have $- H_{a_0} \tilde \chi \gtrsim \chi$, and also the sums of squares 
representation 
\[
- H_{a_0} \tilde \chi = \phi_1^2 +\phi_2^2, \qquad \chi = \phi_3^2
\]
for smooth nonnegative symbols $\phi_j \in S^\frac12$.

Now we define 
\[
\tilde q = q + \tilde \chi, \qquad q_{comp} = \chi_{<75R}(|x|) \tilde q.
\]
Here the symbol $\tilde q$ inherits from $q$ all the properties in Proposition~\ref{Doi_construction_pert}, and we note that $\tilde q$ and $q_{comp}$ depend upon $t$ through $q$ and hence through $a$ as above.
This is because outside $2B_{R}$ the symbol $\tilde \chi$ has similar properties, while inside 
$2B_{R}$ the contribution of the symbol $\tilde \chi$ is small compared to that of $q$. The reason 
we introduce $\tilde q$ is related to the low regularity of the coefficients, which causes us to once again replace 
G{\aa}rding's inequality with more robust sums of squares methods. Precisely, we will have a representation of $\tilde q$ and $\nabla \tilde q$ in terms of the above squares,
\[
\tilde q = d_3 \phi_3^2, \qquad (q_x, |\xi| q_\xi) = d_1  \phi_1^2 + d_2 \phi_2^2
\]
with $d_j \in e^{CML} S^0$. The bound on $d_j$'s is exponentially large, but that will suffice later on. One consequence
of it is that
\begin{equation} \label{amu-below}
- H_{a_\mu} q - CM q  \gtrsim e^{-CML} \sum \phi_j^2, \qquad \mu \geq \lambda.
\end{equation}

On the other hand the cutoff $\chi_{<75R}(|x|)$ achieves the goal of having $q_{comp}$ compactly supported, at the price 
of violating the positive commutator condition (3) in Proposition~\ref{Doi_construction_pert} in the region 
$\{ 50 R < |x| < 100 R\}$.

We now show that  the above  choice for the symbol $q_{comp}$  yields the bound \eqref{w-comp}.
We start with the counterpart of \eqref{com-relation}, namely 
\begin{multline}\label{com-relation-R}
 \Re \langle Q_{comp} w, w \rangle (T) +  \int_0^T \langle C w, w \rangle dt =   \Re \langle Q_{comp} w, w \rangle (0) \\+
2 \Re \int_{0}^T  \langle Q_{comp}w, -i f \rangle  
+ \langle Q_{comp,t}w, w \rangle  \, dt
\end{multline}
where $C$ is now given by 
\[
C =  i[P,Q_{comp}]+i Q_{comp} B - i  B^* Q_{comp}.
\]
The terms on the right are easily bounded as follows:
\[
 \Re \langle Q_{comp} w, w \rangle (0) \lesssim e^{CML} \| w(0)\|_{L^2}^2,
\]
respectively 
\[
\Re \int_{0}^T  \langle Q_{comp,t}w, w \rangle  \, dt
\lesssim  e^{CML} \| w\|_{L^2L^2}^2
\]
and
\[
\int_0^T \langle Q_{comp}w, -i f \rangle  \, dt \lesssim \| Q_{comp} w\|_{L^2H^\frac12} \| \chi_{<100R} f\|_{L^2 H^{-\frac12}}
+ T^\frac12 \| w\|_{L^\infty L^2} \| f  \|_{Y^0}, 
\]
where the second term on the right accounts for the smoothing, rapidly
decreasing tails arising from the contribution of $\chi_{>100R} f$. 
The first term on the left
is estimated by G{\aa}rding's inequality, 
\[
 \langle Q_{comp} w, w \rangle (T)  \gtrsim  \| \chi_{<2 R} w (T)\|_{L^2}^2 -   e^{CML} \| w(T)\|_{H^{-1/2}}^2 
\]
where the right hand side is further estimated by \eqref{lot-bound}.

It remains to consider the contribution of $C$, where we again have the
difficulty of having to deal with low regularity coefficients. In
order to deal with this, we first truncate the coefficients at the
scale $\mu = 2^{\kappa_0} > \lambda$, where $\lambda = e^{CML}$ 
is the scale used earlier in the proof of the Proposition. 
Then we consider the
the contribution of $P_\mu$ and $B_\mu$. This is given by 
\[
C_\mu =  i[P_\mu,Q_{comp}]+i Q_{comp} B_\mu - i  B^*_\mu Q_{comp}.
\]
Then $C_\mu \in OPC^1S^1 \cap \mu OPC^2S^1$,
with principal symbol
\[
c_{\mu,0} = - \{ a_\mu, q_{comp} \} + 2q_{comp} \Im b_\mu .
\]
By the above proposition, this satisfies the bound from below
\[
|\xi|^{-1} c_{\mu,0} (x,\xi) \geq c \chi_{<100 R}(|x|) \chi(x,\xi) 
 - e^{CML} \chi_{<100 R}  p_{in}(x,\xi).
\]
Then G{\aa}rding's inequality yields the fixed time bound (see \cite{TataruSumsSquares})
\[
 \langle C_{\mu} w, w \rangle \gtrsim c \| \chi_{<2 R} w\|_{H^\frac12}^2 
- e^{CML}  \| \chi_{<100 R} \win\|_{H^\frac12}^2 - \mu e^{CML} \| w\|_{L^2 L^2}^2.
\]
Here the $\mu$ factor in the last term arises due to the fact that the
$C^2S^1$ symbol regularity is needed for G{\aa}rding's inequality. For later use, we record
another consequence of G{\aa}rding's inequality. Precisely, by \eqref{amu-below}
we get
\[
c_{\mu,0} q   \gtrsim e^{-CML} \sum \phi_j^2
\]
which gives 
\begin{equation}\label{Cmu-below}
 \langle C_{\mu} w, w \rangle \gtrsim  e^{-CML} \sum \|\Phi_j u\|^2_{L^2}  - \mu e^{CML} \| w\|_{L^2}^2.
\end{equation}

It remains to consider the contribution of $C - C_\mu$, for which it suffices 
to prove the bound 
\begin{equation}\label{dcmu}
\langle (C-C_\mu) w, w \rangle \lesssim \mu^{-\delta} e^{CML}(  \langle C_{\mu} w, w \rangle + 
\| \chi_{<100 R} \win\|_{H^\frac12}^2) + \mu  e^{CML} \| w\|_{L^2}^2.
\end{equation}
This suffices provided that $\mu$ is large enough, $\mu = e^{C_1 ML}$ with $C_1 \gg C$.

We first directly  compute the regularity
\[
C - C_\mu \in \mu^{-\delta} OPC^1S^1,
\]
which shows that only the principal symbol of $C-C_\mu$ matters.

We then use the squares representation for $q$ and $\nabla q$ to write
\[
(c - c_\mu)_0 = \sum e_j \phi_j^2 + e_0 |\xi|  (\chi_{<100R} p_{in})^2, \qquad e_j \in \mu^{-\delta} e^{CML} C^1S^0,
\]
which implies the bound
\[
\langle (C-C_\mu) w, w \rangle \lesssim \mu^{-\delta} e^{CML} \left ( \sum \|\Phi_j w\|^2_{L^2} + 
\| \chi_{<100R} w_{in} \|^2_{H^\frac12} + \| w\|_{L^2}^2 \right) .
\]
Combining this with \eqref{Cmu-below} yields  \eqref{dcmu} and completes the proof of \eqref{w-comp}.


\subsection{ The high frequency local energy decay bound}
\label{gluing}

Our objective here is to use  the results from Sections \ref{small}, \ref{incoming} and \ref{compact} in order 
to complete the proof of the high frequency local energy decay bound for the paradifferential equation
in Proposition~\ref{p:le-L2-lot}.

Combining the bounds in \eqref{w-inc} and \eqref{w-comp} we obtain a local energy decay bound 
\begin{equation}\label{le-loc}
\| \chi_{< 4R} w \|_{L^2 H^\frac12} \lesssim e^{C(M)L}  
(\|w_0\|_{L^2} + \|f\|_{Y^0} +  \| w\|_{L^2 L^2}).
\end{equation}
It remains to estimate the exterior part of $w$. For that we simply truncate $w$, setting 
\[
w_{ext} = \chi_{> 2R} w.
\] 
Then we write the paradifferential equation for $w_{ext}$, but using the truncated coefficients $g_{ext}$
and $b_{ext}$. This takes the form
\begin{eqnarray}
\label{lin1_hf}
\left\{ \begin{array}{l}
(i \partial_t + \partial_k T^w_{g_{ext}^{kl}}   \partial_l   + T^w_{b_{ext}}  \cdot\nabla) w_{ext} = f_{ext}   \\ \\
w_{ext} (0) = \chi_{> 2R} w_{0},
\end{array} \right.
\end{eqnarray}
where 
\[
f_{ext} =  \chi_{> 2R} f + [(\partial_k T^w_{g^{kl}_{ext}}   \partial_l   + T^w_{b_{ext}}  \cdot\nabla), \chi_{> 2R}] w + (\partial_k T^w_{g^{kl}_{ext}}\partial_l - \partial_k T^w_{g^{kl}}\partial_l)w_{ext}
+(T^w_{b_{ext}}\cdot\nabla - T^w_{b}\cdot \nabla)w_{ext}.
\]
Using \eqref{Ydef} and the disjointness of the supports of $w_{ext}$ and $g_{ext}-g$, $b_{ext}-b$, it follows that
\[\|(\partial_k T^w_{g^{kl}_{ext}}\partial_l - \partial_k T^w_{g^{kl}}\partial_l)w_{ext}
+(T^w_{b_{ext}}\cdot\nabla - T^w_{b}\cdot \nabla)w_{ext}\|_{Y^0} \lesssim \|w\|_{L^2L^2}.\]
Moreover, 
\[\|[(\partial_k T^w_{g^{kl}_{ext}}   \partial_l   + T^w_{b_{ext}}  \cdot\nabla), \chi_{> 2R}] w\|_{Y^0} \lesssim \|\chi_{<4R} w\|_{L^2H^{1/2}} + \|w\|_{L^2L^2}.\]
Thus we can apply the small data result  \eqref{e-L2alt} to bound $w_{ext}$,
\begin{equation}\label{le-loc1}
\| w_{ext} \|_{X^0} \lesssim_M 
\|w_0\|_{L^2} + \|f\|_{Y^0} + \| w\|_{L^2 L^2} + \|\chi_{<4R} w\|_{L^2H^{1/2}}.
\end{equation}
Combined with \eqref{le-loc}, this yields the conclusion of Proposition~\ref{p:le-L2-lot} for the case of the paradifferential equation.

\subsection{ Higher regularity bounds for the paradifferential equation}
Here we  extend the $L^2$ high frequency bounds \eqref{e-L2+lot} for the paradifferential equation \eqref{linear-para}
 to $H^s$ and $\ell^1 H^s$. We begin by proving the $\ell^1 L^2$ bound,
\begin{equation}\label{e-L1+lot} 
\| w\|_{\ell^1 X^0} \lesssim   e^{C(M) L} ( \|w_0\|_{\ell^1 L^2} + \| f \|_{\ell^1 Y^0} + 
 \| w\|_{\ell^1L^2 L^2 }).
\end{equation}
Here the $\ell^1$ norms control the $\ell^2$ norms, so by \eqref{e-L2+lot} we can bound
$\| w\|_{X^0}$ by the right side of \eqref{e-L1+lot}.  
  This suffices to establish \eqref{e-L1+lot} in $B_{4R}$, and it remains to consider the exterior part of $w$.  Here we apply \eqref{small-l1} to $w_{ext}$, which solves \eqref{lin1_hf}.  Arguing in a fashion analogous to that in the preceding subsection yields \eqref{e-L1+lot}.

Next we consider \eqref{le-sigma-noT}. By
interpolation, it suffices to consider the case when $\sigma$ is a
positive integer. But this case can be obtained simply by
differentiating the paradifferential equation and applying \eqref{e-L1+lot} to
$\partial^\sigma w$. Here we note that $\partial^\sigma w$ can be viewed as a vector valued function, which solves 
a system which is diagonal at leading order but coupled through the first order terms. This makes no difference, as 
the proof of \eqref{e-L1+lot} equally applies to this case without any
changes.  

Finally we prove
\eqref{le-l1-sigma-noT}. Here we apply the same reasoning as in the
previous paragraph but starting with \eqref{e-L2+lot} instead of
\eqref{e-L1+lot}.


\subsection{ Bounds for the original  equation}

Here we transfer the high frequency bounds from the paradifferential equation \eqref{linear-para} to the 
original equation \eqref{linear}, and prove Proposiition~\ref{p:le-L2-lot} (which immediately implies Theorem~\ref{t:le-L2}),
as well as the bounds \eqref{le-sigma-noT} and \eqref{le-l1-sigma-noT} (which immediately imply Corollary~\ref{lincor1}).

We begin with \eqref{le-sigma-noT} and \eqref{le-l1-sigma-noT}.  This requires bounds for the operator 
\[
E = (g^{kl} - T^w_{g^{kl}}) \partial_k \partial_l + (b^j - T^w_{b^j}) \partial_j
\]
namely
\begin{equation}
\| E w\|_{l^2 Y^\sigma} \lesssim_M e^{- CML} \| w\|_{l^2 X^\sigma} +  e^{CML} \| w\|_{l^2 L^2 H^\sigma} ,
\end{equation}
\begin{equation}
\| E w\|_{\ell^1 Y^\sigma} \lesssim_M e^{-CML} \| w\|_{\ell^1 X^\sigma} + e^{CML} \| w\|_{\ell^1 L^2 H^\sigma} .
\end{equation}

Using a Littlewood-Paley decomposition we identify two distinct cases:

\begin{description}
\item[(i)] High $\times$ high $\to$ low interactions. If the high frequency is $k \gg ML$ then we use \eqref{xxy2+},
gaining a factor of $2^{-\tilde \delta k}$. Else we bound the $\ell^1 X^\sigma$ norm by the $\ell^1 L^2 H^\sigma$ norm,
losing an $  e^{CML} $ factor.  Here $\tilde \delta = s-s_0$.

\item[(ii)] High $\times$ low $\to$ high interactions, where we can still apply the same strategy as above.
 If the low frequency is $k \gg ML$ then we use \eqref{xxy2+}, gaining a factor of $2^{-\delta k}$. 
Else we bound the $\ell^1 X^\sigma$ norm by the $\ell^1 H^\sigma$ norm,
losing an $  e^{CML} $ factor.

\end{description}

\section{Proof of 
Theorem \ref{thm:main1}}
\label{sec:proof}

We recall that the equation \eqref{eqn:quasiquad} turns into an
equation of the form \eqref{eqn:quasiquad1} by differentiation. Hence
it suffices to prove part (b) of the theorem.  Thus, we are working with the equation
\begin{equation}
\label{eqn:quasiquad1-re1}
\left\{ \begin{array}{l}
i u_t + \p_j g^{jk} (u,\bar u)  \p_ku = 
F(u,\bar u,\nabla u,\nabla \bar u) , \ u:
\RR \times \RR^d \to \CC^m \\ \\
u(0,x) = u_0 (x)
\end{array} \right. 
\end{equation}
which we rewrite in the paradifferential form
\begin{equation}
\label{eqn:quasi-para2}
\left\{ \begin{array}{l}
i \partial_t u + \p_j T^w_{g^{jk}}  \p_k u + T^w_{b^j} \partial_j  u + 
T^w_{\tilde b^j} \partial_j \bar u =  G
 \\ \\
u(0,x) = u_0 (x).
\end{array} \right. 
\end{equation}
Here the nonlinearity $G = G(u,\bar u, \nabla u, \nabla \bar u)$ is no longer purely algebraic, as it involves 
frequency localizations. However,  $G$ plays a perturbative role, due to the estimates in Section~\ref{sec:G}.

\subsection{ The iteration scheme}

To construct a local solution $u$ for our nonlinear equation \eqref{eqn:quasiquad1-re1} we introduce 
an iterative scheme as follows:

\begin{itemize}
\item Our starting point is the function $u^{(0)} = 0$. 

\item The iteration step is as follows. Given $u^{(n)}$, we construct $u^{(n+1)}$ as the solution to the 
linear paradifferential equation
\begin{equation}
\label{eqn:quasi-para3}
\left\{ \begin{array}{l}
(i \partial_t u^{(n+1)} + \p_j T^w_{g^{jl} (u^{(n)}) } \p_l u^{(n+1)} + T^w_{b^j(u^{(n)})} \partial_j  u^{(n+1)} + 
T^w_{\tilde b^j(u^{(n)})} \partial_j \bar u^{(n+1)} =  G(u^{(n)})
 \\ \\
u^{(n+1)}(0) = u_0.
\end{array} \right. 
\end{equation}
\end{itemize}

A priori we do not even know whether the sequence $u^{(n)}$ is well
defined for all $n$.  Even if it is defined locally in time, we do not
know whether their lifespans are uniformly bounded from below away
from zero. Thus, our objectives will be, in order, as follows:

\begin{itemize}
\item Establish uniform bounds for the sequence $u^{(n)}$ on a fixed time interval $[0,T]$ 
\emph{ not depending on $n$. }

\item Prove convergence for the sequence $u^{(n)}$.
\end{itemize}

To achieve this we first need to establish the main parameters which will be used to control the sequence
$u^{(n)}$, which should depend only on the initial data $u_0$. A priori we know that $u_0 \in \ell^1 H^s$.
We introduce a second Sobolev index $s_0$ so that 
\[
\frac{d}2 + 2 < s_0 < s.
\]
Then we consider the coefficients $g(u_0)$, $b(u_0)$ and $\tilde b(u_0)$ in the linearized equation, and measure 
them at both regularities, denoting 
\begin{equation}
M_s = \| u_0 \|_{\ell^1 H^s}, \qquad M =   \| u_0 \|_{\ell^1 H^{s_0}}.
 \end{equation}
The next step is to choose a large enough ball $B_R$ of radius $R$ so that $u_0$ is small outside $B_R$,
\begin{equation}\label{choose-R}
\| u_0 \|_{\ell^1 H^{s_0} (B_R^c)} \leq \epsilon \ll 1.
\end{equation}
Here $\epsilon$ is a small universal constant.

Finally, the metric $g(u_0)$ is nontrapping. Then we denote by $L$ the length of the longest bicharacteristic 
for $g(u_0)$, measured on the cosphere bundle, from the entry to the exit from $2B_R$. 

Given the initial data parameters, $M,R,L$ and $M_s$ we seek to use them in order to 
uniformly describe the  sequence $u^{(n)}$:

\begin{prop}\label{p:it-uniform}
Assume that the time $T$ is small enough,
\begin{equation}\label{ub-T1}
T \ll_{M_s} e^{-C(M) L} .
\end{equation}
Then the sequence $u^{(n)}$ is well defined in $[0,T]$ for all $n$ and satisfies the following uniform
properties:
\begin{enumerate}
\item Uniform $H^s$ bounds: 
\begin{equation}\label{ub-hs}
\| u^{(n)} \|_{\ell^1 X^{s} [0,T] } \leq e^{C(M) L} M_s.
\end{equation}
\item Uniform $H^{s_0}$ bounds:
\begin{equation}\label{ub-hs0}
\| u^{(n)} \|_{\ell^1 X^{s_0}[0,T] } \leq  2M.
\end{equation}

\item Uniform exterior size:
\begin{equation}\label{ub-ext}
\| u^{(n)} \|_{\ell^1 X^{s_0} ([0,T] \times B_R^c )} \leq  2\epsilon.
\end{equation}

\item  Uniform nontrapping: 
\begin{equation}\label{ub-L}
L(u^{(n)}) \leq 2L.
\end{equation}
\end{enumerate}
\end{prop}

Once we have the uniform bounds on the iterations, the next goal is to prove convergence 
in a weaker topology:

\begin{prop}\label{p:it-converge}
Assume that the time $T$ is small enough,
\begin{equation}
T \ll_{M_s} e^{-C(M) L} .
\end{equation}
Then the sequence $u^{(n)}$ converges in $\lX^{\sigma}$
for $0 \leq  \sigma < s_0-1$. 
\end{prop}

After these two propositions are proved, it follows that the sequence 
$u^{(n)}$ is uniformly bounded in $\ell^1 X^s$ and convergent in $\lX^{s-1}$. Then it is 
convergent in all intermediate topologies. This suffices in order to pass to the limit in the equation
and conclude that in the limit we obtain a solution $u \in \lX^s$ for the original equation:

\begin{prop}\label{local-exist}
Let $u_0 \in \lH^s$ be a nontrapping  initial datum with parameters $R,M,L,M_s$. 
Assume that the time $T$ is small enough,
\begin{equation}\label{ub-T}
T \ll_{M_s} e^{-C(M) L} .
\end{equation}
Then there exists a solution $u \in \lX^s$ with the following properties:
\begin{enumerate}
\item $H^s$ bound: 
\begin{equation}\label{ub-hs+}
\| u \|_{\ell^1 X^{s} [0,T] } \leq e^{C(M) L} M_s.
\end{equation}
\item $H^{s_0}$ bounds:
\begin{equation}\label{ub-hs0+}
\| u \|_{\ell^1 X^{s_0}[0,T] } \leq  2M.
\end{equation}

\item Small exterior size:
\begin{equation}\label{ub-ext+}
\| u \|_{\ell^1 X^{s} ([0,T] \times B_R^c) } \leq  2\epsilon.
\end{equation}

\item  Nontrapping: 
\begin{equation}\label{ub-L1}
L(u) \leq 2L.
\end{equation}
\end{enumerate}
\end{prop}

The aim of the next two subsections is to prove the first two propositions.

\subsection{The iteration scheme: uniform bounds}

Our aim here is to prove Proposition~\ref{p:it-uniform}. We use  induction on $n$.

\emph{Proof of \eqref{ub-hs}.}  Here we use the bounds for the paradifferential equation in Corollary~\ref{lincor1} 
together with the bounds for $G$ in Proposition~\ref{p:G}
to get
\[
\begin{split}
\| u^{(n+1)}\|_{\ell^1 X^{s}[0,T] } \lesssim & \ e^{C(M) L} ( \| u_0\|_{\ell^1 H^{s}} + \| G(u^{(n)}) \|_{l^1Y^s})
\\ 
\lesssim & \ e^{C(M) L}  ( M_s  + T^\delta C(M) \|u^{(n)} \|_{ \ell^1 X^{s}[0,T] })  .
\end{split}
\]
Hence for small enough $T$ as in \eqref{ub-T} the bound \eqref{ub-hs} follows.

\bigskip

\emph{Proof of \eqref{ub-hs0}.} 
A direct time integration in our iteration yields
\[
\|u^{(n+1)}\|_{\lX^0} \lesssim \sqrt{T} e^{C(M)L}M_s .
\]
Interpolating with the $l^1 X^s$ bound above yields
\begin{equation}\label{u-s0}
\| u^{(n+1)} \|_{\ell^1 X^{s_0}} \lesssim  T^{\delta/2} e^{C(M)L} M_s,
\end{equation}
which suffices for $T$ as in \eqref{ub-T}.

\bigskip

\emph{Proof of \eqref{ub-ext}.} 
This follows directly from the  
bound \eqref{u-s0} in view of the choice of $R$ in \eqref{choose-R}.

\bigskip

\emph{Proof of \eqref{ub-L}.}  This is a consequence of Proposition~\ref{p:nontrap2} due to  
\eqref{u-s0}.

\subsection{The iteration scheme: weak convergence}

Here we prove Proposition~\ref{p:it-converge}, which asserts
 that our iteration scheme converges in the weaker $l^1 X^\sigma$ topology. We recall the range for $\sigma$,
 namely
\begin{equation}\label{sigma-range}
0 \leq \sigma < s_0-1.
\end{equation}
For this we write an equation for the difference
 $v^{(n)} = u^{(n+1)} - u^{(n)}$:
\begin{equation}
\label{iterate-diff}
\left\{ \begin{array}{l}
(i \partial_t  +  \p_j T^w_{g^{jk, (n)}}    \p_k  +  T^w_{b^{(n)}} \nabla) v^{(n)}
+  T^w_{\tilde b^{(n)}} \nabla \bar v^{(n)} = G(u^{(n)}) - G(u^{(n-1)}) + H^{(n)}
\\ \\
v^{(n)}(0,x) = 0,
\end{array} \right. 
\end{equation}
where 
\[
H^{(n)} =   \p_j T^w_{g^{jk, (n)}-g^{jk,(n-1)}}    \p_k u^{(n)} +  T^w_{b^{(n)}-b^{(n-1)}} \nabla u^{(n)}
+  T^w_{\tilde b^{(n)}-\tilde b^{(n-1)}} \nabla \bar u^{(n)}.
\]
For the $G$ difference we apply Proposition~\ref{p:G},
which yields
\begin{equation}
\|   G(u^{(n)}) - G(u^{(n-1)}) \|_{l^1 Y^\sigma} \lesssim_M
T^\delta \| v^{(n-1)} \|_{l^1X^\sigma}.
\end{equation}
Similarly, for $H^{(n)}$ we claim the bound
\begin{equation}\label{Hnbd}
\|  H^{(n)} \|_{l^1 Y^\sigma} \lesssim_M
T^\delta \| v^{(n-1)} \|_{l^1 X^\sigma}.
\end{equation}
Assume this holds, then for the
paradifferential equation we use Corollary~\ref{lincor1}. We obtain
\[
\| v^{(n+1)}\|_{\lX^{\sigma}} \lesssim T^\delta e^{C(M)L} \| v^{(n)}\|_{\lX^{\sigma}}, 
\]
which for $T$ as in \eqref{ub-T} yields
\[
\| v^{(n+1)}\|_{\lX^{\sigma}} \lesssim \frac12 \| v^{(n)}\|_{\lX^{{\sigma}}} .
\]
The desired convergence follows. Thus we have established the existence part of our main theorem.

It remains to prove the bound \eqref{Hnbd}.
We write the coefficients above in the form
\[
g^{jk, (n)}-g^{jk,(n-1)} = v^{(n-1)} h_0(u)
\]
respectively 
\[
b^{(n)}-b^{(n-1)} = h_1(u,\nabla u) \nabla v^{(n-1)}
+ h_2(u,\nabla u) v^{(n-1)}
\]
where $u$ stands for $(u^{(n)}, u^{(n-1)})$.
For the functions $h_j$ we can apply the Moser estimates
in \eqref{moser}.  
Then we are left with proving 
the following two trilinear bounds,
\begin{equation}\label{xxxy1}
\| S_{<k-4} (wv) u_k\|_{\lY^\sigma} \lesssim T^\delta  
\| w\|_{\lX^{s_0-1}} \| v\|_{\lX^\sigma} \| u\|_{\lX^{s_0-2}},
\end{equation}
respectively 
\begin{equation}\label{xxxy2}
\| S_{<k-4} (wv) u_k\|_{\lY^\sigma} \lesssim T^\delta  
\| w\|_{\lX^{s_0-1}} \| v\|_{\lX^{\sigma-1}} \| u\|_{\lX^{s_0-1}}.
\end{equation}
We consider two cases:

a) The $v$ frequency is $\leq 2^{k-2}$. Then 
the frequency of $w$ is similar, and the first bound 
\eqref{xxxy1} is worse.  We harmlessly drop the multiplier, use \eqref{xxy1+} for the $vu$ product 
and bound $w$ in $L^\infty$.

b) The $v$ frequency is $\geq 2^{k-2}$. Then 
the frequency of $w$ is also similar, and 
the worst case is \eqref{xxxy2} with  $\sigma = 0$.
But then this is exactly the bound \eqref{tri}, proved earlier.

\subsection{ Uniqueness via weak Lipschitz dependence}

Our aim here is to prove the following estimate for the difference of two solutions:

\begin{prop}
Let  $u^{(1)}_0 \in \ell^1 H^s$ be a nontrapping initial datum with parameters $M,R,L$ and $M(s)$.
Let $u^{(2)}_0\in \ell^1 H^s$ be another initial datum with comparable size,
\[
\| u^{(2)}_0\|_{\ell^1 H^s} \lesssim M_s
\]
and close to $u^{(1)}_0 $ in a weak topology,
\[
\|u^{(1)}_0- u^{(2)}_0\|_{\ell^1 L^2} \ll_{M_s} e^{-C(M)L}.
\]
 Then the following hold:

a) $u^{(2)}_0$ has comparable parameters, and the same associated ball $B$. 

b) The associated solutions exist on the time interval $[0,T]$ with $T$ as in \eqref{ub-T}.

c) The following difference estimate holds in $[0,T]$
for $\sigma$ as in \eqref{sigma-range}:
\begin{equation}\label{weak-lip}
\|  u^{(1)} - u^{(2)}\|_{\lX^{\sigma}} \lesssim e^{C(M) L} \|  u^{(1)}(0) - u^{(2)}(0)\|_{l^1 H^\sigma}.
\end{equation}
\end{prop}
Uniqueness follows as a corollary of this result.
\begin{proof}
a) Interpolating between the $\ell^1 L^2$ and $\ell^1 H^s$ bounds we get
\[
 \|u^{(1)}_0- u^{(2)}_0\|_{\ell^1 H^{s_0}} \ll  e^{-C(M)L}.
\]
Thus $M$ and $B_R$ remain the same. Finally, $L$ remains the same by Proposition~\ref{p:nontrap2}.

b) This is a consequence of Proposition~\ref{local-exist}.

c) This repeats the arguments in the proof of Proposition~\ref{p:it-converge}.
\end{proof}

\subsection{ Frequency envelopes and higher regularity}

Here we consider solutions with initial data $u_0$ as in
Proposition~\ref{local-exist} and provide bounds for the frequency
envelope of the solutions in terms of the frequency envelope of the
initial data. As a corollary, we show that higher regularity for the
data implies higher regularity for the solution.

Our starting point is the initial data $u_0\in \ell^1 H^s$, for which we consider 
an admissible frequency envelope $c_k$. Then our frequency envelope bound for the solutions is 

\begin{prop}\label{p:fe}
Let $u$ be a solution to \eqref{eqn:quasiquad1-re} as in Proposition~\ref{local-exist}, with initial data 
$u_0 \in \ell^1 H^s$. Let $c_k$ be an admissible frequency envelope for $u_0$ in $ \ell^1 H^s$.
Then  the solution $u$ satisfies the dyadic bounds
\begin{equation}
\| S_k u\|_{\ell^1 X^s} \lesssim e^{C(M) L} c_k .
\end{equation}
\end{prop}

As a consequence, we obtain the following higher regularity statement:

\begin{cor}\label{cor:higher}
Let $u$ be a solution to \eqref{eqn:quasiquad1-re} as in Proposition~\ref{local-exist}, with initial data $u_0$.
Assume in addition that $u_0 \in \ell^1 H^\sigma$ for some $\sigma \geq s$. Then 
the solution $u$ satisfies the bound
\begin{equation}
\| u\|_{\lX^\sigma} \lesssim e^{C(M) L} \|u_0\|_{l^1 H^\sigma}, \qquad \sigma \geq s.
\end{equation}
\end{cor}

\begin{proof}[Proof of Proposition~\ref{p:fe}]
We denote by $d_k$ a \emph{minimal} admissible frequency envelope for $u$ in $l^1X^s$.
By Corollary~\ref{c:G} 
we have a corresponding bound for $G(u)$, namely
\[
\| S_k G(u)\|_{\lX^s} \lesssim_M T^\delta d_k.
\]
Now we apply to $u$ the bound for the linear paradifferential equation in Corollary~\ref{fe-est}.
This gives
\[
\| S_k u \|_{\lX^s} \lesssim e^{C(M) L} ( c_k + T^\delta d_k).
\]
The envelope on the right is admissible, so by the minimality of $d_k$ we obtain
\[
d_k \lesssim e^{C(M) L} ( c_k + T^\delta d_k).
\]
Now the choice of $T$ as in \eqref{ub-T} guarantees that
\[
d_k \lesssim e^{C(M) L}  c_k 
\]
as needed.
\end{proof}

\begin{proof}[Proof of Corollary~\ref{cor:higher}]
Let $c_k$ be a minimal admissible frequency envelope for the initial data $u_0$ in $\ell^1 H^s$.
Then
\[
\| u_0\|_{\lH^\sigma}^2 \approx \sum 2^{2(\sigma - s)k} c_k^2. 
\]
On the other hand, the previous proposition implies that
\[
\| u\|_{\lX^\sigma}^2 \lesssim e^{C(M) L} \sum 2^{2(\sigma - s)k} c_k^2 \lesssim   e^{C(M) L}\| u_0\|_{\lH^\sigma}^2.
\]
\end{proof}

\subsection{ Continuous dependence on the initial data}
\label{sec:contdep}
Here we want to show that $u_0 \to u$ is continuous from $l^1 H^s$
into $\lX^s$.  The argument is exactly as in \cite{MMT3}.

Let $a^{(n)}_j$ and $a_j$ be minimal frequency envelopes given by \eqref{freqEnv} for $u_0^{(n)}$ and $u_0$ in $l^1H^s$.  If $u_0^{(n)} \to u_0$ in $l^1 H^s$, then $(a_j^{(n)})\to (a_j)$ in $l^2$.  So for any $\epsilon>0$, there is $N_\epsilon$ so
\[
\| a^{(n)}_{> N_\epsilon} \|_{l^2} \leq \epsilon \qquad \text{for all $n$}.
\]

We remark that our initial data convergence in $l^1 H^s$ guarantees that, for 
large enough $n$, the control parameters $R,L$ can be uniformly chosen independently of $n$.  Then the associated solutions $u^{(n)}$ exist on a 
uniform time interval $[0,T]$. By Proposition~\ref{p:fe}, the envelopes $a_j^{(n)}$ carry over to the solutions $u^{(n)}$ measured in $\lX^s[0,T]$.
In particular,  we conclude that
\begin{equation}\label{hf-diff}
\| u^{(n)}_{> N_\epsilon} \|_{\lX^s} \leq \epsilon \qquad \text{for all $n$}.
\end{equation}

Using \eqref{weak-lip} for low frequencies 
and \eqref{hf-diff} for the high frequencies, we obtain
\[
\begin{split}
\| u^{(n)} - u\|_{\lX^s} \lesssim & \| S_{< N_\epsilon} (u^{(n)} - u)\|_{\lX^s} 
+  \| S_{> N_\epsilon} u^{(n)} \|_{\lX^s} +  \| S_{> N_\epsilon} u \|_{\lX^s} 
\\
\lesssim &  2^{N_\epsilon} \| S_{< N_\epsilon} (u^{(n)} - u)\|_{\lX^{s-1}} 
+  2 \epsilon
 \\
\lesssim &  2^{N_\epsilon} \| S_{< N_\epsilon} (u^{(n)}_0 - u_0)\|_{l^1 H^{s-1}} 
+  2 \epsilon.
\end{split}
\]
As $n \to \infty$ we see that
\[
\lim \sup_{n \to \infty} \| u^{(n)} - u\|_{\lX^s} \lesssim \epsilon.
\]
So upon letting $\epsilon \to 0$,
\[
\lim_{n \to \infty} \| u^{(n)} - u\|_{\lX^s} = 0,
\]
which gives the desired result.


\begin{thebibliography}{XX}

\bibitem{Bej1} I. Bejenaru: Quadratic nonlinear derivative
    Schr\"odinger equations. I.  {\em IMRP Int. Math. Res. Pap.} {\bf
    2006}, 84pp.

\bibitem{Bej} I. Bejenaru:  Quadratric nonlinear derivative
    Schr\"odinger equations. II. {\em Trans. Amer. Math. Soc.} {\bf 360}, No. 11 (2008), 5925--5957.

\bibitem{BT} I. Bejenaru and D. Tataru: Large data local solutions for the
derivative NLS equation.  {\em J. Eur. Math. Soc.} {\bf 10} (2008), 957--985.


\bibitem{Blair} M. Blair:  Strichartz estimates for wave equations with coefficients of Sobolev regularity.  {\em Communications in Partial Differential Equations} {\bf 31}, No. 5 (2006), 649--688.


\bibitem{BRV} O. Blasco, A. Ruiz, and L. Vega: Non interpolation
    in Morrey-Campanato and block spaces.  {\em Ann. Scuola
  Norm. Sup. Pisa Cl. Sci.} {\bf 28} (1999), 31--40.

\bibitem{Ch} H. Chihara: Local existence for semilinear
    Schr\"odinger equations.  {\em Math.  Jap.}  {\bf 42} (1995), 35--52.

\bibitem{CS} P. Constantin and J.-C. Saut: Local smoothing
    properties of dispersive equations.  {\em J. Amer. Math. Soc.} {\bf 1}
  (1989), 413--446.

\bibitem{CKS} W. Craig, T. Kappeler, and W. Strauss:  Microlocal
    dispersive smoothing for the Schr\"odinger equation. {\em Comm. Pure
  Appl. Math.} {\bf 48}, No. 8 (1995), 769--860.

\bibitem{D1} S. Doi: Remarks on the Cauchy problem for
    Schr\"odinger-type equations. {\em  Comm. Partial Differential Equations} {\bf 21} (1996), 163--178.

\bibitem{Doi} S. Doi: Smoothing effects for Schr\"odinger
    evolution equation and global behavior of geodesic flow.  {\em  Math. Ann.} {\bf 318} (2000), 355--389.

\bibitem{HO} N. Hayashi and T. Ozawa: Remarks on nonlinear
    Schr\"odinger equations in one space dimension.
  {\em Differential Integral Equations} {\bf 7} (1994), 453--461.

\bibitem{Hor} L. H\"ormander:  {\em Lectures on nonlinear hyperbolic
    differential equations}, {\it Mathematiques \& Applications},
  Berlin (1997).



\bibitem{Ich} W. Ichinose: On $L^2$ well-posedness of the Cauchy
    problem for Schr\"odinger type equations on a Riemannian manifold
    and Maslov theory. {\em Duke Math. J.} {\bf 56} (1988), 549--588.



\bibitem{Kato} T. Kato: On the Cauchy problem for the
    (generalized) Korteweg-de Vries equation. {\em Advances in
  Math. Supp. Studies, Studies in Applied Math.} {\bf 8} (1983), 93--128.

\bibitem{KPV2} C. E. Kenig, G. Ponce, and L. Vega: Small
    solutions to nonlinear Schr\"odinger equations.
 {\em  Ann. Inst. H. Poincar\'e Anal. Non Lin\'eaire} {\bf 10} (1993), 255--288.

\bibitem{KPV3} C. E. Kenig, G. Ponce, and L. Vega:  Smoothing
    effects and local existence theory for the generalized nonlinear
    Schr\"odinger equations.  {\em Invent. Math.} {\bf 134} (1998), 489--545.

\bibitem{KPV} C. E. Kenig, G. Ponce, and L. Vega: The Cauchy problem for
quasi-linear Schr\"odinger equations. {\em Invent. Math.} {\bf 158} (2004), 343--388.

\bibitem{KPRV1} C. E. Kenig, G. Ponce, C. Rolvung, and L. Vega:  The
    general quasilinear ultrahyperbolic Schr\"odinger equation.
 {\em  Adv. Math.} {\bf 196}, No. 2 (2005), 402--433.

\bibitem{KPRV2} C. E. Kenig, G. Ponce, C. Rolvung, and L. Vega: Variable
coefficient Schr\"odinger flows for ultrahyperbolic operators. {\em Adv. Math.} {\bf 206}, No. 2 (2006), 373--486.

\bibitem{KT} H. Koch and D. Tataru: Dispersive estimates for
    principally normal pseudodifferential operators. {\em  Comm. Pure App.
  Math.} {\bf 58} (2005), 217--284.

\bibitem{KrFa} S. N. Krushkov and A. V. Faminiskii:  Generalized
    solutions of the Cauchy problem for the Korteweg-de Vries
    equation.  {\em Mat. Sh. USSR} {\bf 48} (1984), 93--138.

\bibitem{LimPon} W.-K. Lim and G. Ponce: On the initial value
    problem for the one dimensional quasilinear Schr\"odinger
    equation.  {\em  SIAM J. Math. Anal.} {\bf 34} (2003), 435--459.

\bibitem{LinPon} F. Linares and G. Ponce: {\em Introduction to
    nonlinear dispersive equations}. {\it Universitext}, Springer, New
  York (2009).

\bibitem{MMT1} J. Marzuola, J. Metcalfe, and D. Tataru: Wave packet
parametrices for evolutions governed by pdo's with rough
symbols.  {\em Proc. Amer. Math. Soc.} {\bf 136}, No. 2 (2007), 597--604.

\bibitem{MMT2} J. Marzuola, J. Metcalfe, and D. Tataru: Strichartz estimates
and local smoothing estimates for asymptotically flat Schr\"odinger equations.
 {\em J. Funct. Anal.} {\bf 255}, No. 6 (2008), 1497--1553.

\bibitem{MMT3} J. Marzuola, J. Metcalfe, and D. Tataru: Quasilinear Schrödinger equations I: Small data and quadratic interactions.  {\em Adv. Math.} {\bf 231} (2012), no. 2, 1151--1172.

\bibitem{MMT4} J. Marzuola, J. Metcalfe, and D. Tataru:  Quasilinear Schrödinger equations, II: Small data and cubic nonlinearities. {\em Kyoto J. Math.} {\bf 54} (2014), no. 3, 529--546.

\bibitem{Michalowski} N. Michalowski: Quasilinear Schr\"odinger equations.  Preprint, 2014.  (ArXiv: 1410.0057)

\bibitem{Miz1} S. Mizohata: Some remarks on the Cauchy problem. {\em  J. Math.
Kyoto Univ.} {\bf 1} (1961), 109--127.

\bibitem{Miz2} S. Mizohata: Sur quelques equations du type Schr\"odinger.
{\em  Journees Equations aux derivees partielles}, Saint-Jean de Monts, 1981.

\bibitem{Miz3} S. Mizohata: On the Cauchy Problem. {\em Notes and Reports in Mathematics in Science and Engineering}, {\bf 3}, Science Press and Academic Press (1985).

\bibitem{Sch} T. Schottdorf: Ill-posedness for the quadratic D-NLS equation. Master's
Thesis, University of California at Berkeley (2010).

\bibitem{Sj} P. Sj\"olin: Regularity of solutions to the
    Schr\"odinger equations.  {\em Duke Math. J.} {\bf 55} (1987), 699--715.

\bibitem{Sog} C. D. Sogge: {\em Lectures on nonlinear wave equations},
  {\it Monographs in Analysis, II}. International Press, Boston (1995).

\bibitem{StTa} G. Staffilani and D. Tataru: Strichartz estimates for a Schr\"odinger operator with nonsmooth coefficients.  {\em Comm. PDE} {\bf 27}, No. 7-8
(2002), 1337--1372.

\bibitem{Tak80} J. Takeuchi:  On the Cauchy problem for some non-Kowalewskian equations with distinct characteristic roots.  {\em J. Math. Kyoto Univ.}, {\bf 20}, No. 1 (1980), 105-124.

\bibitem{lp} D. Tataru: Strichartz estimates for second order hyperbolic
operators with nonsmooth coefficients III.  {\em J. Amer. Math. Soc.} {\bf 15}
(2002), 419--442.

\bibitem{TataruSumsSquares} D. Tataru: On the Fefferman-Phong inequality and related problems.  {\em Comm. PDE}, {\bf 27}, Issue 11-12 (2002), 2101--2138.

\bibitem{T} D. Tataru: Phase space transforms and microlocal
 analysis. Phase space analysis of partial differential equations.
  Vol. II.  {\em Pubbl. Cent. Ric. Mat. Ennio Giorgi, Scuola Norm. Sup.
  Pisa} (2004), 505--524.

\bibitem{T-WM} D. Tataru: Rough solutions for the wave maps
    equation.  {\em Amer. J. Math.} {\bf 127}
(2005), 293--377.

\bibitem{Tat} D. Tataru: Parametrices and dispersive estimates for Schr\"odinger operators with variable coefficients.  {\em  Amer. J. Math.} {\bf 130} (2008), 571--634.
  
\bibitem{Taylor-nl}  M.E. Taylor: {\em Tools for PDE: pseudodifferential operators, paradifferential operators, and layer potentials}. No. 81. American Mathematical Soc. (2007).

\bibitem{VV} A. Vargas and L. Vega: Global wellposedness for 1{D} non-linear Schr\"odinger equation for data with an infinite $L^2$
    norm.  {\em J. Math. Pures Appl.} {\bf 80} (2001), 1029--1044.

\bibitem{Vega} L. Vega: The Schr\"odinger equation: pointwise
    convergence to the initial data.  {\em Proc. Amer. Math. Soc.} {\bf
    102} (1988), 874--878.
\end{thebibliography}
\end{document}